\documentclass{amsart}
\usepackage{graphicx} 
\usepackage[utf8]{inputenc}
\usepackage[utf8]{inputenc}
\usepackage{amsfonts}
\usepackage{amsmath}
\usepackage{amssymb}
\usepackage{mathrsfs}
\usepackage{amsthm}
\usepackage{float}
\usepackage{upgreek}
\usepackage{enumitem}
\usepackage{mathabx}
\usepackage{overpic}
\usepackage{color}
\usepackage{thmtools}
\usepackage{thm-restate}
\usepackage{stmaryrd}
\usepackage[dvipsnames]{xcolor}
\usepackage{caption}
\usepackage{subcaption}
\usepackage{thmtools}
\usepackage{enumitem}
\usepackage{thm-restate}
\usepackage{esint}
\usepackage{xcolor}
\newcommand{\subsetsim}{\mathrel{\ooalign{\raise.4ex\hbox{$\subset$}\cr$\raise-.9ex\hbox{$\sim$}$}}}
\declaretheorem[name=Corollary,numberwithin=section]{cor}
\usepackage[left=1.03in, right=1.03in, top=1.05in, bottom=1.05in]{geometry}
\usepackage{hyperref}
\usepackage{cleveref}
\numberwithin{equation}{section}

\newcommand{\p}{\mathcal{P}}

\newcommand{\T}{\mathcal{T}}
\DeclareMathOperator{\Z}{\mathbb{Z}}
\DeclareMathOperator{\J}{\mathcal{J}}

\DeclareMathOperator{\A}{\mathcal{A}}

\newcommand{\N}{\mathcal{N}}
\newcommand{\K}{\mathcal{K}}
\DeclareMathOperator\supp{supp }

\newcommand{\dd}[1]{\mathrm{d}{#1}}
\newcommand{\diam}{\operatorname{diam}}
\DeclareMathOperator{\R}{\mathbb{R}}
\DeclareMathOperator{\B}{\mathbb{B}}

\newtheorem{theorem}{Theorem}[section]
\newtheorem{lemma}[theorem]{Lemma}

\newtheorem{prop}[theorem]{Proposition}
\newtheorem{corollary}[theorem]{Corollary}
\newtheorem{remark}[theorem]{Remark}
\theoremstyle{definition}
\newtheorem{definition}[theorem]{Definition}
\theoremstyle{definition}
\newtheorem{example}[theorem]{Example}
\newtheorem{notation}[theorem]{Notation}

\usepackage{graphicx}

\newtheorem{claim}[theorem]{Claim}

\DeclareMathOperator{\dist}{dist}

\newcommand{\qad}{\mathrm{dim}_{\mathrm{qA}}}

\newcommand{\I}{\mathscr{I}}
\title{Nikod\'ym maximal function with restricted directions}
\thanks{Both authors are supported by the European Research Council (ERC) under the European Union’s Horizon Europe research and innovation programme (grant agreement No 101087499). T.O. is also supported by the Research Council of Finland via the project \emph{Approximate incidence geometry}, grant no. 355453.}
\author{Tuomas Orponen and Hrit Roy}
\begin{document}
\begin{abstract} We study the planar Nikod\'ym maximal operator $\mathcal{N}_{\Theta;\delta}$ associated to a direction set $\Theta \subset \mathbb{S}^{1}$. We show that the quasi-Assouad dimension $s := \dim_{\mathrm{qA}} \Theta$ characterises the essential $L^{p}$-boundedness of $\mathcal{N}_{\Theta;\delta}$ in the following sense. If $s \in [\tfrac{1}{2},1]$, then $\mathcal{N}_{\Theta;\delta}$ is essentially bounded on $L^{p}(\R^{2})$ for $p \geq 1 + s$, and essentially unbounded for $p < 1 + s$. Here essential boundedness means $L^{p}$-boundedness with constant $O_{\epsilon}(\delta^{-\epsilon})$. We also show that the characterisation described above fails for $s < \tfrac{1}{2}$. More precisely, there exists a set $\Theta \subset \mathbb{S}^{1}$ with $\dim_{\mathrm{qA}}  \Theta = \tfrac{1}{3}$ such that $\mathcal{N}_{\Theta;\delta}$ is essentially unbounded on $L^{p}(\mathbb{R}^{2})$ for all $p < \tfrac{3}{2}$.
    
    As an application, we show there exists a convex domain with affine dimension $\tfrac{1}{6}$ such that the $\alpha$-order Bochner--Riesz means converge in $L^6$ for all $\alpha>0$. 

\end{abstract}
\maketitle
\section{Introduction}
\label{sec: introduction}
\subsection{Nikod\'ym maximal operator}
 For $\delta\in (0,1)$, the classical \textit{Nikod\'ym maximal operator} is defined as \begin{equation}\label{eqn: nikodym def}
  \mathcal{N}_{\delta}f(x):=\sup_{x(T)=x}\fint_T |f(y)|\,\mathrm{d}y  \quad\text{for all}\quad f\in L^1_{\mathrm{loc}}(\R^d),
\end{equation} where the supremum is taken over all $\delta$-tubes $T$ with centre $x(T)$ at $x$. A $\delta$-tube is the closed $(\delta/2)$-neighbourhood of a unit line segment. It is well-known that $(\N_\delta)_{\delta>0}$ are not uniformly bounded in $L^p$ for $p\neq\infty$, due to the existence of zero-measure \textit{Nikod\'ym sets}. However, the \textit{Nikod\'ym maximal conjecture} asserts that for all $d\leq p\leq\infty$,
\begin{equation}
    \label{eqn: nikodym essential boundedness}
    \|\N_\delta f\|_{L^p(\R^d)}\lesssim_\epsilon\delta^{-\epsilon}\|f\|_{L^p(\R^d)}\quad\text{for all}\quad \epsilon>0.
\end{equation}
The planar version of this conjecture was resolved by C\'ordoba in 1977:
\begin{theorem}[C\'ordoba \cite{Cordoba_Nikodym}]
\label{theorem: cordoba}
    For all $p\geq 2$, we have
    \begin{equation}
        \label{eqn: L^2 bound}
        \|\N_\delta f\|_{L^p(\R^2)}\lesssim(\log\delta^{-1})^{1/p}\|f\|_{L^p(\R^2)}.
    \end{equation}
\end{theorem}
The Nikod\'ym maximal conjecture is equivalent to the \textit{Kakeya maximal conjecture} (discussed below), and is still open in dimensions $d\geq 3$. 

In this paper, we will consider the following variant of the Nikod\'ym maximal operator in the plane. For a set of directions $\Theta\subseteq\mathbb{S}^1$, we define 
\begin{equation}
    \label{eqn: fractal nikodym def}
    \mathcal{N}_{\Theta;\delta} f(x):=\sup_{\substack{x(T)=x,\\\omega(T)\in\Theta}}\fint_T |f(y)|\,\mathrm{d}y  \quad\text{for all}\quad f\in L^1_{\mathrm{loc}}(\R^2),
\end{equation}
where the supremum is restricted to those $\delta$-tubes with direction $\omega(T)\in\Theta$ and centre $x(T)$ at $x$. Thus, $\mathcal{N}_{\delta} = \mathcal{N}_{\mathbb{S}^{1};\delta}$. If $\Theta$ is finite, then the operators $\N_{\Theta;\delta}$, $\delta > 0$, are uniformly $L^p$-bounded for all $p \in [1,\infty]$ with constant $|\Theta|$ (the cardinality of $\Theta$). This follows from
\begin{equation}\label{form59} \mathcal{N}_{\Theta;\delta}f(x) \leq \sum_{\theta \in \Theta} \fint_{x + T_{\theta}} |f(y)| \, \mathrm{d}y, \qquad x \in \R^{2}, \end{equation}
where $T_{\theta}$ is the $\delta$-tube centred at the origin, with direction $\theta$. However, the constant $|\Theta|$ may not be optimal. The optimal constant in the $p=2$ case has been studied for general direction-sets and equidistributed direction-sets (see \cite{stromberg_annals, katz, katz2, ASV}, and \cite{demeter_3d_nikodym,kim,diplinio--parissis} for higher dimensions). The operators $\mathcal{N}_{\Theta;\delta}$ are also uniformly bounded for certain ``sparse" but infinite direction-sets, see \cite{alfonseca,bateman,CF_differentiation,NSW,SS,stromberg_thesis}, and \cite{parcet--rogers} for higher dimensions. For example, if $\Theta$ consists of directions with slopes in $2^{-\mathbb{N}}$, then uniform boundedness holds for all $p>1$. It is believed that the direction sets for which uniform boundedness holds for some (all) $1<p<\infty$ are precisely those which can be written as finite unions of finite order lacunary sets.\footnote{This is the main result in \cite{bateman}. However, some anomalies were recently observed in the methods of \cite{bateman}, which has cast some doubts to the validity of the result: see \cite{hagelstein}.} These sets have Assouad (and box and Hausdorff) dimension zero. 

In this paper, we will focus on direction-sets that have positive dimension, and work with the weaker notion of boundeness as in \eqref{eqn: nikodym essential boundedness} that allows $\epsilon$-power losses. Let us define the critical exponent for the Nikod\'ym maximal operator to be $$p_\Theta:=\inf\big\{p:\|\N_{\Theta;\delta}\|_{L^p\to L^p}\lesssim_\epsilon\delta^{-\epsilon}\;\forall\epsilon>0\big\}.$$

\begin{remark}\label{rem1} The ``$\inf$" in the definition of $p_{\Theta}$ is a ``$\min$". In other words, we claim that $$\|\mathcal{N}_{\Theta;\delta}\|_{L^{p_{\Theta}} \to L^{p_{\Theta}}} \lesssim_{\epsilon} \delta^{-\epsilon}\quad\text{for all}\quad \epsilon > 0.$$
To see why, fix $\epsilon > 0$ and $f \in L^{p_{\Theta}}(\R^{2})$. First, note that it is sufficient to show that if $Q \subset \R^{2}$ is a square of unit side-length, then 
\begin{equation}\label{form61} \|\mathcal{N}_{\Theta;\delta}f\|_{L^{p_{\Theta}}(Q)} \lesssim_{\epsilon} \|f\|_{L^{p_{\Theta}}(\R^{2})}. \end{equation} 
Indeed, we may decompose $\R^{2}$ into a disjoint union of (dyadic) unit squares $Q$, and for each $Q$ it holds $\mathbf{1}_{Q}\mathcal{N}_{\Theta,\delta}f = \mathbf{1}_{Q}\mathcal{N}_{\Theta,\delta}(f\mathbf{1}_{3Q})$. Now, we apply H\"older's inequality with exponent $p > p_{\Theta}$ satisfying $p_{\Theta}^{-1} - p^{-1} < \epsilon/4$ to estimate
\begin{equation}\label{form60} \|\mathcal{N}_{\Theta;\delta}f\|_{L^{p_{\Theta}}(Q)} \leq \|\mathcal{N}_{\Theta;\delta}f\|_{L^{p}(\R^{2})}. \end{equation}
The next easy observation is that the following estimate holds pointwise: $\mathcal{N}_{\Theta;\delta}f \lesssim \mathcal{N}_{\Theta,2\delta}(|f| \ast \varphi_{\delta})$, where $\varphi_{\delta}(x) = \delta^{-2}\varphi(x/\delta)$ is a standard approximate identity. Continuing from \eqref{form60}, and using $p > p_{\Theta}$, 
\begin{displaymath} \|\mathcal{N}_{\Theta;\delta}f\|_{L^{p_{\Theta}}(Q)} \lesssim_{\epsilon} \delta^{-\epsilon/2}\||f| \ast \varphi_{\delta}\|_{L^{p}(\R^{2})}. \end{displaymath}
Young's inequality implies $\||f| \ast \varphi_{\delta}\|_{L^{p}(\R^{2})} \leq \|f\|_{L^{p_{\Theta}}(\R^{2})}\|\varphi_{\delta}\|_{L^{q}(\R^{2})}$ for $p^{-1} + 1 = p_{\Theta}^{-1} + q^{-1}$, thus $q^{-1} - 1 = p^{-1} - p_{\Theta}^{-1} < 0$. Here $\|\varphi_{\delta}\|_{L^{q}(\R^{2})} \lesssim \delta^{2(q^{-1} - 1)} = \delta^{2(p^{-1} - p_{\Theta}^{-1})} \leq \delta^{-\epsilon/2}$. Chaining these estimates proves \eqref{form61}. \end{remark}

Searching for the infimal exponent of (essential) boundedness is sensible, because trivially $\|\mathcal{N}_{\mathbb{S}^{1};\delta}\|_{L^{\infty} \to L^{\infty}} \leq 1$. Theorem \ref{theorem: cordoba} shows that $p_\Theta\leq 2$ for all $\Theta\subseteq\mathbb{S}^1$. In the extreme case where the upper box dimension $\overline{\dim}_{\mathrm{B}} \Theta=0$, it is easily seen from an estimate analogous to \eqref{form59} (but summing only over a $\delta$-net $\Theta_{\delta} \subset \Theta$) that $p_{\Theta} = 1$. Previously, it was not known whether $p_\Theta<2$ for any direction set with positive dimension. Our main result provides a characterisation of $p_{\Theta}$ in terms of the \textit{quasi-Assouad dimension} of $\Theta$, denoted $\qad \Theta$. This notion of dimension was introduced in \cite{qA_dimension}, and the definition is recalled in Section \ref{sec: maximal estimate}.
\begin{theorem}
    \label{theorem: main}
    Let $\Theta\subseteq\mathbb{S}^1$ be a set with $\qad \Theta = s\in [1/2,1]$.
    Then $p_{\Theta} = 1 + s$. In other words,
    \begin{equation}
        \label{eqn: main}
        \|\N_{\Theta;\delta} f\|_{L^{p}(\R^2)}\lesssim_\epsilon\delta^{-\epsilon}\|f\|_{L^{p}(\R^2)}\quad\text{for all $\epsilon>0$,}
    \end{equation}
    if and only if $p\geq 1+s$.
\end{theorem}
The numerology $1 + s$ is not altogether unexpected: in 1995, Duoandikoetxea and Vargas \cite[Theorem 1]{duo--vargas} proved that if $\overline{\dim}_{\mathrm{B}} \Theta \leq s$, then the operators $\mathcal{N}_{\Theta;\delta}$ are uniformly bounded (without $\delta^{-\epsilon}$-losses) in the space $L^{p}_{\mathrm{rad}}(\R^{2})$ of radial $L^{p}$-functions for $p > 1 + s$. Theorem \ref{theorem: main} tells us that the upper box dimension of $\Theta$ is not the correct notion to consider in the absence of radiality: in fact, for every $\epsilon > 0$, there exist sets $\Theta \subset \mathbb{S}^{1}$ with $\overline{\dim}_{\mathrm{B}} \Theta \leq \epsilon$ and $\qad \Theta = 1$ (see \cite[Proposition 1.6]{qA_dimension}), so Theorem \ref{theorem: main} tells us that $p_{\Theta} = 2$. In fact, the lower bound $p_{\Theta} \geq 1 + \qad \Theta$ holds for all values of $\qad \Theta$ in $[0,1]$, see Proposition \ref{prop: sharpness}.

Let us record the following evident corollary of Theorem \ref{theorem: main} in the case $s<1/2$.
\begin{corollary}\label{cor: 3/2 bound}
    Let $\Theta\subset\mathbb{S}^1$ be a set with $\qad \Theta = s\in[0,1/2)$. Then $p_{\Theta} \leq 3/2$.\end{corollary} 
In the range $\qad \Theta \in [0,1/2)$, our results are not complete, but at least Corollary \ref{cor: 3/2 bound} cannot be improved to match Theorem \ref{theorem: main}: Proposition \ref{ex3} exhibits a set $\Theta \subset \mathbb{S}^{1}$ with $\qad(\Theta) \leq 1/3$ such that $p_{\Theta} \geq 3/2$. The following upper bound may still be plausible: $p_{\Theta} \leq 1 + \dim_{\mathrm{A}} \Theta$, where $\dim_{\mathrm{A}}$ is the \emph{Assouad dimension}. In particular, if $\Theta$ is $s$-Ahlfors regular for $s \in [0,\tfrac{1}{2}]$, is it true that $p_{\Theta} \leq 1 + s$?

\begin{remark}
    Strictly speaking, the papers \cite{stromberg_annals, katz, katz2, ASV, demeter_3d_nikodym,kim,diplinio--parissis, CF_differentiation,SS,stromberg_thesis,alfonseca,NSW,bateman, parcet--rogers} mentioned above pertain to the stronger \textit{directional maximal operator} $$M_\Theta f(x):=\sup_{\omega\in\Theta}\int_{-1/2}^{1/2} |f(x+t\omega)|\,\dd{t},$$ rather than the Nikod\'ym maximal operator. As explained by Demeter in \cite[\S8]{demeter_3d_nikodym}, averaging over lines yields a more singular operator than averaging over $\delta$-tubes, which has an additional smoothing effect. While the Nikod\'ym maximal operator is controlled by the directional maximal operator, the converse may not be true.
\end{remark}
\subsection{A variant for the Kakeya maximal operator}\label{s:Kakeya} In this section we discuss variants of the previous problem for the \textit{Kakeya maximal operator}, defined by $$\mathcal{K}_{\delta}f(\omega):=\sup_{\omega(T)=\omega}\fint_T|f(y)|\,\dd{y}, \qquad f\in L^1_\mathrm{loc}(\R^d),$$ where the supremum is taken over all $\delta$-tubes with direction $\omega\in\mathbb{S}^{d-1}$. The Kakeya maximal conjecture asserts that for all $d\leq p\leq\infty$, 
$$\|\K_{\delta}f\|_{L^p(\mathbb{S}^{d-1})}\lesssim_\epsilon\delta^{-\epsilon}\|f\|_{L^p(\R^d)}\quad\text{for all $\epsilon>0$}.$$ This conjecture is equivalent to the Nikod\'ym maximal conjecture via point-line duality. Interpolating the Kakeya analogue of \eqref{eqn: L^2 bound} with the trivial $L^1$-bound, we have the following sharp result in $\R^2$.
\begin{theorem}
    [C\'ordoba \cite{Cordoba_Nikodym}]\label{theorem: cordoba kakeya}
    For all $p\leq 2$, we have
    \begin{equation}
        \label{eqn: kakeya l2}
        \|\K_\delta f\|_{L^p(\mathbb{S}^1)}\lessapprox\delta^{1-2/p}\|f\|_{L^2(\R^2)}.\footnote{Here, the notation $A\lessapprox B$ means that $A\leq C(\log(1/\delta))^{C}B$, for some absolute constant $C>0$.}
    \end{equation}
\end{theorem}
Similar to the Nikod\'ym case, we can try to restrict the Kakeya maximal operator to fractal direction sets. To make this precise, let $\mu$ be a Borel measure on $\mathbb{S}^1$ satisfying the \textit{Frostman condition} $$\mu(B(\omega,r))\leq Cr^s\quad\text{for all}\quad \omega\in\mathbb{S}^1,\,r>0,$$
for some $C>0$ and $s\in[0,1]$. Such a measure is called an \textit{$s$-Frostman measure}. We would like to obtain sharp estimate for $\|\K_\delta\|_{L^p(\R^2)\to L^p(\dd{\mu})}$ for a general $s$-Frostman measure $\mu$. First, we observe that for all $p$, \eqref{eqn: kakeya l2} is the best possible estimate for every non-zero measure $\mu$ on $\mathbb{S}^1$. Indeed, we have $\K_\delta\mathbf{1}_{B(0,\delta)}\gtrsim\delta$, so $$\frac{\|\K_\delta \mathbf{1}_{B(0,\delta)}\|_{L^p(\dd{\mu})}}{\|\mathbf{1}_{B(0,\delta)}\|_{L^p(\R^2)}}\gtrsim\delta^{1-2/p}\mu(\mathbb{S}^1)^{1/p}.$$ 
On the other hand, if $\mu$ is \textit{$s$-Ahlfors regular}, that is, there exist $C,c>0$ such that 
$$cr^s\leq\mu(B(\omega,r))\leq Cr^s\quad\text{for all}\quad \omega\in\mathbb{S}^1,\,r>0,$$ then \eqref{eqn: kakeya l2} fails for $p>1+s$. Indeed, \cite[Proposition 3.1]{mitsis} shows that 
$$\|\K_\delta\|_{L^p(\R^2)\to L^p(\dd{\mu})}\gtrsim \delta^{-\frac{1-s}{p}}\quad\text{for all $p\geq 1$.}$$ In particular, for $p>1+s$ we get $\|\K_\delta\|_{L^p(\R^2)\to L^p(\dd{\mu})}\gtrsim\delta^{1-2/p-\eta}$, where $\eta=1-\frac{1+s}{p}>0$. 
\begin{restatable}{theorem}{kakeya}\label{theorem: fractal kakeya}
      Let $\mu$ be an $s$-Frostman measure on $\mathbb{S}^1$ for $s\in[0,1]$. For all $p\leq 1+s$, it holds
    \begin{equation*}
        \|\K_\delta f\|_{L^p(\dd{\mu})}\lessapprox\delta^{1-2/p}\|f\|_{L^p(\R^2)}.
    \end{equation*}
\end{restatable}
The discussion above demonstrates that the power on the right-hand side, as well as the range of $p$ is sharp. Theorem \ref{theorem: fractal kakeya} is much easier than Theorem \ref{theorem: main}, and follows by a short computation from a maximal function bound due to J. Zahl \cite[Theorem 1.2]{zahl_fractal_kakeya}. In fact, Zahl's result is valid in all dimensions (and his result is highly non-trivial in this generality), but it turns out that the planar case is rather simple. We thank K. Rogers for a permission to include his short argument, see Lemma \ref{lemma: rogers}.

\begin{remark} Mitsis \cite{mitsis} proved that if $\mu$ is an $s$-Frostman measure on $\mathbb{S}^{1}$, then
  \begin{displaymath}
                 \|\K_\delta f\|_{L^p(\dd{\mu})}\lesssim \delta^{-\frac{1-s}{p}}\|f\|_{L^p(\R^2)},\quad\text{for all}\quad p\geq 2.
     \end{displaymath}
 Up to a multiplicative error $C(\log(1/\delta))^{C}$, this result follows by interpolating the bound in Theorem \ref{theorem: fractal kakeya} with the $L^{\infty}$-bound $\|\mathcal{K}_{\delta}f\|_{L^{\infty}(\mu)} \leq \|f\|_{L^{\infty}(\R^{2})}$. \end{remark}

Finally, we remark that although Theorem \ref{theorem: cordoba} is equivalent to \ref{theorem: cordoba kakeya}, there is no apparent implication between Theorem \ref{theorem: main} and Theorem \ref{theorem: fractal kakeya}. This is because restricting the directions $\omega$ for the Kakeya maximal operator would translate to restricting the \emph{positions} $x$ for the Nikod\'ym maximal operator, and vice-versa.
\subsection{An incidence theorem}
Theorem \ref{theorem: main} will be obtained as a consequence of an incidence bound, which may be of independent interest. In order to state the theorem, we need the following terminology from incidence geometry.

\begin{notation}[Dyadic cubes and tubes] In the following, for $\delta \in 2^{-\mathbb{N}}$ and $A \subset \R^{d}$, the notation $\mathcal{D}_{\delta}(A)$ refers to the family of dyadic $\delta$-cubes intersecting $A$. We abbreviate $\mathcal{D}_{\delta} := \mathcal{D}_{\delta}([0,1)^{d})$. 

The notation $\mathcal{T}^{\delta}$ refers to \emph{dyadic $\delta$-tubes in $\R^{2}$} of side $\delta \in 2^{-\mathbb{N}}$. Thus, a generic element $T \in \mathcal{T}^{\delta}$ has the form $\cup \mathbf{D}(p)$, 
where $p \in \mathcal{D}_{\delta}([-1,1) \times \R)$ is a dyadic $\delta$-square contained in $[-1,1) \times \R$, and $\mathbf{D}$ is the following ``point-line duality map":
\begin{displaymath} \mathbf{D}(a,b) := \{(x,y) \in \R^{2} : y = ax + b\}. \end{displaymath}
For $T = \cup \mathbf{D}([a,a + \delta) \times [b,b + \delta))$, the \emph{slope of $T$} is the number $\sigma(T) := a \in \delta \Z \cap [-1,1)$. For a more comprehensive introduction to dyadic $\delta$-tubes, see \cite[Section 2.3]{MR4718435}. \end{notation}

\begin{notation}[Covering number] Let $E$ be a bounded set in $\mathbb{R}^d$. For $r>0$, we denote the $r$-covering number of $E$ by $|E|_{r}$. Thus, $|E|_{r}$ is the minimal number of closed $r$-balls needed to cover $E$. \end{notation}

Recall the definition of \emph{Katz--Tao $(\delta,t,C)$-sets}:

\begin{definition}[Katz--Tao $(\delta,t,C)$-set] Fix $C,t \geq 0$. A set $P \subset \R^{d}$ is called a \emph{Katz--Tao $(\delta,t,C)$-set} if 
\begin{displaymath} |P \cap B(x,r)|_{\delta} \leq C(r/\delta)^{t}, \qquad x \in \R^{d}, \, \delta \leq r \leq 1. \end{displaymath}
A family of dyadic cubes $\mathcal{P} \subset \mathcal{D}_{\delta}(\R^{d})$ is called a Katz--Tao $(\delta,t,C)$-set if $P := \cup \mathcal{P}$ is a Katz--Tao $(\delta,t,C)$ set. A family of dyadic tubes $\mathcal{T} \subset \mathcal{T}^{\delta}$ is called a Katz--Tao $(\delta,t,C)$-set if $\mathcal{T} = \{\cup \mathbf{D}(p) : p \in \mathcal{P}\}$, where $\mathcal{P} \subset \mathcal{D}_{\delta}([-1,1] \times \R)$ is a Katz--Tao $(\delta,t,C)$-set. \end{definition}

We will use the following notion of \emph{$(\delta,s,C)$-regularity}:

\begin{definition}[$(\delta,s,C)$-regular set] Let $C,s > 0$ and $\delta \in 2^{-\mathbb{N}}$. A set $\Theta \subset \delta \Z$ is called \emph{$(\delta,s,C)$-regular} if 
\begin{displaymath} |\Theta \cap I|_{r} \leq C(R/r)^{s}, \qquad \delta \leq r \leq R \leq 1, \, I \in \mathcal{D}_{R}(\R). \end{displaymath}
\end{definition}
\begin{theorem}\label{theorem: incidence bound} For every $C_{\mathrm{KT}},C_{\mathrm{reg}} \geq 1$, $\epsilon > 0$ and $s \in [\tfrac{1}{2},1]$, the following holds for $\delta \in 2^{-\mathbb{N}}$ small enough. Let $\mathcal{T} \subset \mathcal{T}^{\delta}$ be a Katz--Tao $(\delta,1,C_{\mathrm{KT}})$-set such that $\sigma(\mathcal{T})$ is $(\delta,s,C_{\mathrm{reg}})$-regular. Then,
\begin{equation}\label{form31} |P_{r}(\mathcal{T})|_{\delta} \lesssim_\epsilon \delta^{-\epsilon} (C_{\mathrm{KT}}C_{\mathrm{reg}})^{1/s}\frac{\delta^{-1}|\mathcal{T}|}{r^{(s + 1)/s}}, \qquad r \geq 1. \end{equation}
Here $P_{r}(\mathcal{T}) := \{p \in [0,1]^{2} : |\{T \in \mathcal{T} : p \in T\}| \geq r\}$.
\end{theorem}
\begin{remark} Theorem \ref{theorem: incidence bound} is also applicable for $s \in [0,\tfrac{1}{2})$. In fact, if $\sigma(\mathcal{T})$ is $(\delta,s,C_{\mathrm{reg}})$-regular with $s < \tfrac{1}{2}$, then $\sigma(\mathcal{T})$ is also $(\delta,\tfrac{1}{2},C_{\mathrm{reg}})$-regular, and Theorem \ref{theorem: incidence bound} yields
\begin{displaymath} |P_{r}(\mathcal{T})|_{\delta} \lesssim_\epsilon \delta^{-\epsilon} (C_{\mathrm{KT}}C_{\mathrm{reg}})^{2}\frac{\delta^{-1}|\mathcal{T}|}{r^{3}}, \qquad r \geq 1, \, \epsilon > 0. \end{displaymath}  \end{remark} 
\begin{remark} Theorem \ref{theorem: incidence bound} is false if $\sigma(\mathcal{T})$ is merely assumed to be a Katz--Tao $(\delta,s)$-set. Indeed, the dependence on $C_{\mathrm{reg}}$ in \eqref{form31} is sharp, see Section \ref{sec: sharp example}.\end{remark} 

\begin{remark} The reader may wish to compare Theorem \ref{theorem: incidence bound} with several recent incidence bounds obtained recently by Demeter and O'Regan \cite{2025arXiv251115899D,2026arXiv260101264D}. To the best of our knowledge, the results do not overlap, but both proofs are based on an application of the \emph{$2$-ends Furstenberg set theorem} due to H. Wang and S. Wu \cite{2024arXiv241108871W}. \end{remark} 

\subsection{Application to Bochner--Riesz means}\label{sec: BR intro}
Let $\Omega\subset\R^2$ be an open, bounded, convex set with $0 \in \Omega$. We call such a set a \textit{convex domain}. Let $\rho$ be the associated Minkowski functional, which is defined as \begin{equation*}\rho(\xi):=\inf\{t>0:\xi\in t\Omega\},\qquad\xi\in\R^2.\end{equation*}
The Bochner--Riesz operator of order $\alpha\geq 0$ associated with $\Omega$ is given by $$(B^\alpha_\Omega f)\;\widehat{}\;(\xi):=(\max\{1-\rho(\xi),0\})^\alpha \cdot \widehat{f}(\xi),\qquad\xi\in\R^2.$$
Let $\kappa_\Omega\in[0,1/2]$ be the \textit{affine dimension} of $\Omega$ as defined in \cite{SZ} (the definition is also recalled in Section \ref{sec: bochner--riesz}). Seeger and Ziesler \cite{SZ} obtained the following bounds for a general convex domain $\Omega$.
\begin{theorem}[Theorem 1.1 \cite{SZ}]
\label{theorem: seeger--ziesler}
 Let $\Omega$ be a convex domain. Then $B^\alpha_\Omega$ is bounded on $L^p(\R^2)$ for  
  \begin{align*}
  \alpha>0\qquad&\text{when}\qquad 2\leq p\leq 4,\\\label{eqn:sz range}
  \alpha>\kappa_\Omega(1-4/p)\qquad&\text{when}\qquad 4\leq p\leq\infty.
  \end{align*}
\end{theorem}
In \cite{SZ}, Seeger and Ziesler also showed that for all $\kappa\in[0,1/2]$, there exists a domain $\Omega$ with $\kappa_\Omega=\kappa$, for which the bounds in Theorem \ref{theorem: seeger--ziesler} are sharp (see also \cite{cladek_BR} for the sharpness example). However, it was shown by Cladek \cite{cladek_BR} that given $\kappa\in(0,1/6]$, one can construct a domain $\Omega$ with $\kappa_\Omega=\kappa$, such that $B^\alpha_\Omega$ satisfies improved bounds. The range $(0,1/6]$ was subsequently extended to $(0,1/2)$ by the second author \cite{roy}. As a corollary of Theorem \ref{theorem: main}, we obtain further improvements. We state specifically the result with $\kappa_\Omega=1/6$.
Define the critical exponent for Bochner--Riesz as $$p_\Omega:=\sup\big\{p:\|B^\alpha_\Omega\|_{L^p\to L^p}<\infty \text{ for all } \alpha>0\big\}.$$ Theorem \ref{theorem: seeger--ziesler} shows that $p_\Omega\geq 4$ for all convex domains $\Omega$. Compare this with the fact that $\|B^0_{B(0,1)}\|_{L^p\to L^p}=\infty$ for all $p\neq 2$ \cite{Fefferman}. In the extreme case when the affine dimension $\kappa_\Omega=0$, it also shows that $p_\Omega=\infty$ (these domains include convex $n$-gons).
 As a consequence of Theorem \ref{theorem: main}, we obtain the following.
\begin{restatable}{theorem}{B}\label{theorem: B}
    There exists a convex domain $\Omega \subset \R^{2}$ with $\kappa_\Omega=1/6$, satisfying $p_\Omega\geq 6$.
\end{restatable}
It was previously not known whether $p_\Omega>4$ for any domain of positive affine dimension. The question whether such domains exist was raised in \cite{cladek_BR}.

We also have the following quantitative version of Theorem \ref{theorem: B}.
\begin{restatable}{theorem}{A}\label{theorem: A}
For all $\epsilon>0$, there exists a convex domain $\Omega = \Omega_{\epsilon} \subset \R^{2}$ with $\kappa_\Omega=1/6$, such that $B^\alpha_\Omega$ is bounded on $L^p(\R^2)$ for  
  \begin{align*}
  \alpha>\epsilon\qquad&\text{when}\qquad 2\leq p\leq 6,\\
  \alpha>\kappa_\Omega(2/3-4/p)+\epsilon\qquad&\text{when}\qquad 6\leq p\leq 18,\\
  \alpha>\kappa_\Omega(1-10/p)+\epsilon\qquad&\text{when}\qquad 18\leq p\leq\infty.
  \end{align*}
\end{restatable}
\begin{figure}[h]
    \centering
    \makebox[0pt]{\includegraphics[width=11cm]{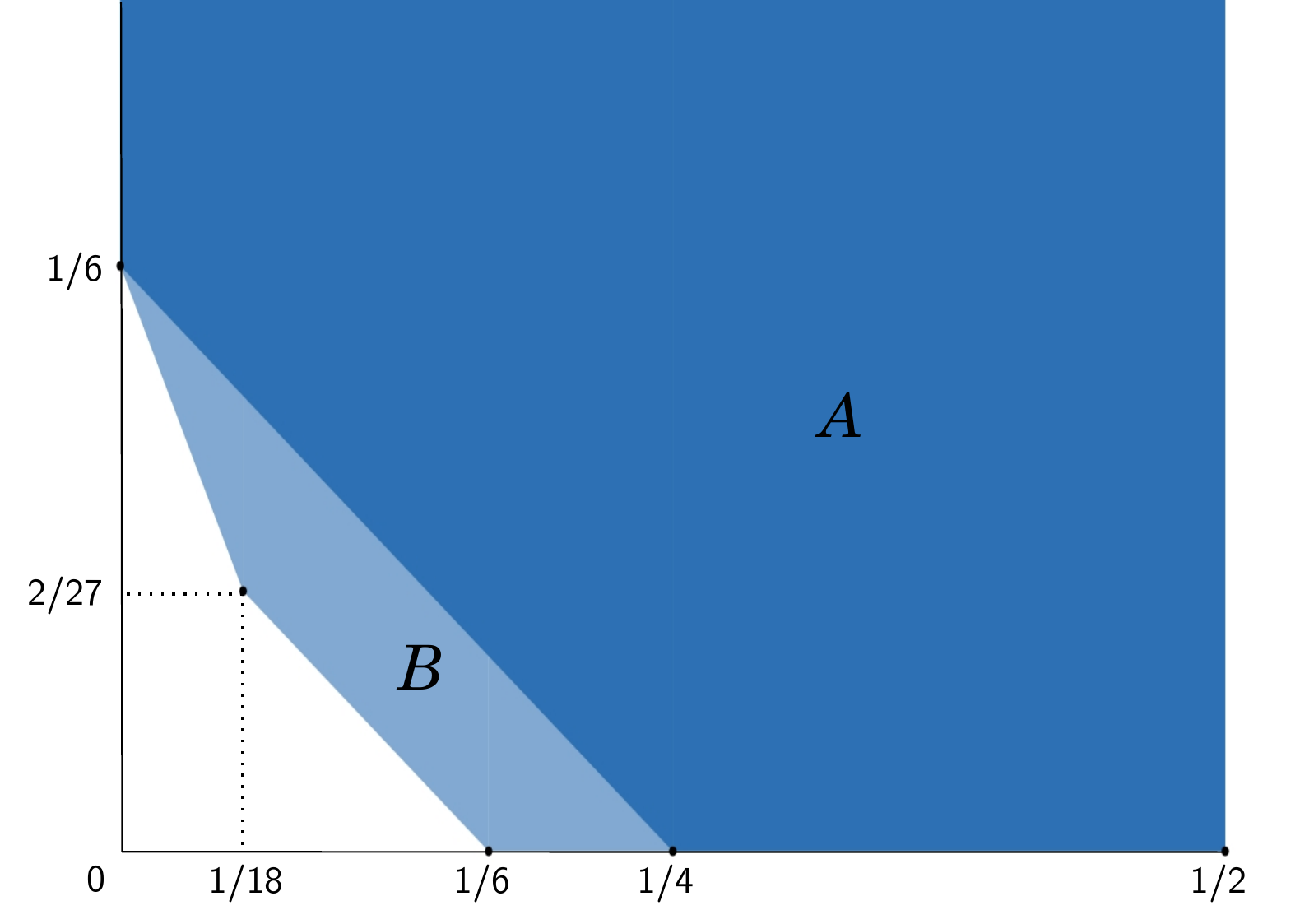}}
    \caption{A comparison of Theorem \ref{theorem: seeger--ziesler} and Theorem \ref{theorem: A}. Theorem \ref{theorem: seeger--ziesler} says that $B^\alpha_\Omega$ is bounded on $L^p$ for all $(1/p,\alpha)\in A$ for any $\Omega$ with $\kappa_\Omega=1/6$. Theorem \ref{theorem: A} constructs a specific $\Omega$ with $\kappa_\Omega=1/6$ such that $B^\alpha_\Omega$ is bounded on $L^p$ for all $(1/p,\alpha)\in A\cup B$.} 
    \label{fig: BR}
\end{figure}

\subsection{Paper outline} The main work in the paper goes into proving the incidence bound, Theorem \ref{theorem: incidence bound}. At its heart, the proof is based on the recent \emph{$2$-ends Furstenberg set estimate} of Wang and Wu \cite{2024arXiv241108871W}, recalled in Theorem \ref{t:WW}. Once Theorem \ref{theorem: incidence bound} has been established in Section \ref{sec: incidence theorem}, we apply it in Section \ref{sec: maximal estimate} to derive Theorem \ref{theorem: main}, our main result on the Nikod\'ym maximal function. Section \ref{sec: maximal estimate} also contains the construction of an example showing that Theorem \ref{theorem: main} fails for $\qad \Theta < \tfrac{1}{2}$, see Proposition \ref{ex3}; that construction relies on work of the first author with K. Ren \cite{2024arXiv241104528O}. Section \ref{sec: bochner--riesz} contains the applications to the convergence of Bochner-Riesz means introduced in Section \ref{sec: BR intro}. Finally, the proof of Theorem \ref{theorem: fractal kakeya} can be found in Appendix \ref{app:kakeya}; we decided to place the proof in an appendix, since the result is essentially a restatement of the planar case of an existing incidence bound due to Zahl, as explained below Theorem \ref{theorem: fractal kakeya}.

\subsection{Notation} Throughout the paper, we shall adopt the following notation, unless otherwise specified. All logarithms shall be assumed as base 2. The notation $A\lesssim B$ or $B\gtrsim A$ shall mean that $A\leq CB$ for some absolute constant $C>0$. Likewise, the notation $A\lesssim_\epsilon B$ or $B\gtrsim_\epsilon A$ shall mean that $A\leq C_\epsilon B$ for some constant $C_\epsilon>0$. The notation $A\lessapprox B$ or $B\gtrapprox A$ shall mean that $A\leq C(\log(1/\delta))^{C}B$, for some absolute constant $C>0$. The notation $A\sim B$ shall be used when both $A\lesssim B$ and $B\lesssim A$ hold. Likewise, $A\sim_\epsilon B$ shall be used when both $A\lesssim_\epsilon B$ and $B\lesssim_\epsilon A$ hold. For a finite set $E$, its cardinality will be denoted by $|E|$, and for a measurable set $F$, we denote its Lebesgue measure by $\mathrm{Leb}(F)$.

\subsection*{Acknowledgments} We would like to thank Amlan Banaji for suggesting and providing references on quasi-Assouad dimension. As already mentioned in Section \ref{s:Kakeya}, we thank Keith Rogers for a permission to include his proof of Lemma \ref{lemma: rogers}.

\section{The incidence theorem}
\label{sec: incidence theorem}
In this section, we prove Theorem \ref{theorem: incidence bound}. The proof will apply the incidence theorem of Wang and Wu \cite[Theorem 2.1]{2024arXiv241108871W} (see also \cite[Theorem 1.5]{2025arXiv250921869W} for a more general version). To state the theorem we need a few further defintions.
\begin{definition}[Frostman $(\delta,t,C)$-set]
Fix $C,t\geq 0$. A set $P\subset\R^d$ is called a \emph{Frostman $(\delta,t,C)$-set} if 
$$|P\cap B(x,r)|_\delta\leq Cr^{-s}|P|_\delta,\qquad x\in\R^d,\,\delta\leq r\leq 1.$$
\end{definition}
Wang and Wu's theorem is stated in terms of Katz--Tao $(\delta,1)$-sets of ``ordinary" (non-dyadic) $\delta$-tubes. It could be equivalently stated in terms of dyadic tubes, but in fact we will find it more convenient to apply the original formulation within the proof of Theorem \ref{theorem: incidence bound}. An \emph{ordinary $\delta$-tube} is any rectangle of dimensions $\delta \times 1$ contained in $\R^{2}$. A family $\mathcal{T}$ of ordinary $\delta$-tubes is called a Katz--Tao $(\delta,t,C)$-set if the following holds: for all $\delta \leq r \leq 1$, and for all rectangles $R \subset \R^{2}$ of dimensions $r \times 2$, it holds
\begin{displaymath} |\{T \in \mathcal{T} : T \subset R\}| \leq C(r/\delta)^{t}. \end{displaymath} 
Two ordinary $\delta$-tubes $T_{1},T_{2} \subset \R^{2}$ are \emph{distinct} if $\mathrm{Leb}(T_{1} \cap T_{2}) \leq \delta/2 = \tfrac{1}{2}\mathrm{Leb}(T_{1})$.
\begin{theorem}\label{t:WW} Let $C,l \geq 1$ and $0 < \epsilon < \eta$. Then the following holds for all $\delta \in 2^{-\mathbb{N}}$ small enough, depending only on $\epsilon$. Let $\mathcal{T}$ be a Katz--Tao $(\delta,1,C)$-set of distinct ordinary $\delta$-tubes. For each $T \in \mathcal{T}$, let $P_{T} \subset \mathcal{D}_{\delta}$ be a family with the properties
\begin{enumerate}[label=(\roman*)]
\item $p \cap T \neq \emptyset$ for all $p \in P_{T}$.
\item $P_{T}$ is a Frostman $(\delta,\eta,\delta^{-\eta^{2}/2})$-set. 
\item $|P_{T}| \sim l$ for all $T \in \mathcal{T}$.
\end{enumerate}
Then,
\begin{equation}\label{form45} \Big| \bigcup_{T \in \mathcal{T}} P_{T} \Big| \gtrsim_{\epsilon} C^{-1}\delta^{(1 + \epsilon + \eta)/2}l^{3/2}|\mathcal{T}|. \end{equation}
\end{theorem}

\begin{remark} The two-ends condition in \cite[Theorem 2.1]{2024arXiv241108871W} is a little more general than (ii). It postulates that there exist constants $0 < \eta_{2} < \eta_{1} < 1$ and $C \geq 1$ such that
\begin{equation}\label{form42} |P_{T} \cap B(x,\delta^{\eta_{1}})| \leq C\delta^{\eta_{2}}|P_{T}|, \qquad x \in \R^{2}. \end{equation}
Note that (ii) implies
\begin{displaymath} |P_{T} \cap B(x,\delta^{\eta})| \leq \delta^{-\eta^{2}/2} \delta^{\eta^{2}}|P_{T}| = \delta^{\eta^{2}/2}|P_{T}|, \end{displaymath}
so \eqref{form42} holds with $(C,\eta_{1},\eta_{2}) := (1,\eta,\eta^{2}/2)$.
\end{remark}

\begin{remark} In \cite[Theorem 2.1]{2024arXiv241108871W}, it is assumed that the Katz--Tao constant $C$ is absolute. However, for $C \geq 1$, a Katz--Tao $(\delta,1,C)$-set $\mathcal{T}$ of ordinary $\delta$-tubes contains a Katz--Tao $(\delta,1,O_{d}(1))$-subset $\mathcal{T}'$ of cardinality $|\mathcal{T}'| \gtrsim_{d} C^{-1}|\mathcal{T}|$, see for example \cite[Lemma 2.2]{2025arXiv251105091O}. Now \eqref{form45} follows by applying \cite[Theorem 2.1]{2024arXiv241108871W} to any such subset $\mathcal{T}'$. \end{remark}

\begin{proof}[Proof of Theorem \ref{theorem: incidence bound}] By a slight abuse of notation, we will denote by $P_{r}(\mathcal{T})$ a maximal $\delta$-separated subset inside $P_{r}(\mathcal{T})$, so our task is to estimate $|P_{r}(\mathcal{T})|$ in place of $|P_{r}(\mathcal{T})|_{\delta}$. 
Let $r(p) := \{T \in \mathcal{T} : p \in T\}$, thus $r(p) \geq r$ for $p \in P_{r}(\mathcal{T})$. For $\mathfrak{r} \geq r$, let 
\begin{displaymath} P(\mathfrak{r}) := \{p \in P_{r}(\mathcal{T}) : r(p) \in [\mathfrak{r},2\mathfrak{r}]\}. \end{displaymath}
Now \eqref{form31} will follow if we manage to prove that
\begin{equation}\label{form40} |P(\mathfrak{r})| \lessapprox \delta^{-\epsilon}(C_{\mathrm{KT}}C_{\mathrm{reg}})^{1/s}\delta^{-1}|\mathcal{T}|/\mathfrak{r}^{(s + 1)/s}, \qquad \mathfrak{r} \geq r. \end{equation}
The notation $A \lessapprox B$ means that $A \leq C_{1}(\log (1/\delta))^{C_{2}}$, where $C_{1} > 0$ may depend on $\epsilon$, and $C_{2} > 0$ is absolute. 

Fix $\mathfrak{r} \geq r$. Note that $\mathfrak{r} \leq C_{\mathrm{reg}}\delta^{-s}$, or equivalently $1 \leq C_{\mathrm{reg}}^{1/s}\mathfrak{r}^{-1/s}\delta^{-1}$, by the $(\delta,s,C_{\mathrm{reg}})$-regularity of $\sigma(\mathcal{T})$. If $|P(\mathfrak{r})| \leq \delta^{-\epsilon}|\mathcal{T}|/\mathfrak{r}$, there is nothing to prove, because in that case 
\begin{displaymath} |P(\mathfrak{r})| \leq \delta^{-\epsilon} \frac{|\mathcal{T}|}{\mathfrak{r}} \leq \delta^{-\epsilon}C_{\mathrm{reg}}^{1/s}\frac{|\mathcal{T}|\delta^{-1}}{\mathfrak{\mathfrak{r}}^{(s + 1)/s}}. \end{displaymath}
So, in the sequel we may assume that
\begin{equation}\label{form11} \mathfrak{r}|P(\mathfrak{r})| \geq \delta^{-\epsilon}|\mathcal{T}|. \end{equation}
Fix $T \in \mathcal{T}$. We will now locate a "two-ends scale" $\Delta_{T}$ for $P(\mathfrak{r}) \cap T$. By \cite[Proposition 2.17]{MR4912925} (applied with constant $C := 1$ and $\eta_{0} := \epsilon/100$), there exists a scale $\Delta_{T} \in 2^{-\mathbb{N}} \cap [\delta,1]$, a number $\eta_{T} \in [\epsilon/100,\epsilon/10]$, a square $Q_{T} \in \mathcal{D}_{\Delta}(P \cap T)$, and a subset $P_{T} \subset P(\mathfrak{r}) \cap T \cap Q_{T}$ with the following properties:
\begin{itemize}
\item[(1) \phantomsection \label{1}] $|P_{T}| \geq \delta^{\eta_{T}}|P(\mathfrak{r}) \cap T|$.
\item[(2) \phantomsection \label{2}] $S_{Q}(P_{T})$ is a Frostman $(\delta/\Delta,\eta_{T},O_{\epsilon}(1))$-set. 
\end{itemize}
We next pass to refinement of $\mathcal{T}$ in such a way that most incidences remain, and the numbers $\eta_{T},\Delta_{T}$ are roughly constant. Before starting, we define the notation
\begin{displaymath} \mathcal{I}(P',\mathcal{T}') := \sum_{p \in P'} |\{T \in \mathcal{T}' : p \in P_{T}\}| = \sum_{T \in \mathcal{T'}} |P' \cap P_{T}|, \qquad P' \subset P, \, \mathcal{T}' \subset \mathcal{T}. \end{displaymath}
In particular, the reader should note that we only count ``restricted" incidences $(p,T)$ with $p \in P_{T}$, which may be a strict subset of those incidences $(p,T)$ with $p \in T$. 

\begin{claim}\label{c1} There exist numbers $\eta \in [\epsilon/100,\epsilon/10]$, $\Delta \in 2^{-\mathbb{N}} \cap [\delta,1]$, and a subset $\mathcal{T}_{0} \subset \mathcal{T}$ such that:
\begin{itemize}
\item[(i)] $\eta_{T} \in [\eta,2\eta)$ and $\Delta_{T} \in [\Delta,2\Delta)$ for all $T \in \mathcal{T}_{0}$.
\item[(ii)] $\mathcal{I}(P(\mathfrak{r}),\mathcal{T}_{0}) \geq \delta^{2\eta}\mathfrak{r}|P(\mathfrak{r})|$.
\end{itemize}
\end{claim} 

\begin{proof} For $\eta \in 2^{-\mathbb{N}} \cap [\epsilon/100,\epsilon/10]$ and $\Delta \in 2^{-\mathbb{N}} \cap [\delta,1]$, write 
\begin{displaymath} \mathcal{T}(\eta,\Delta) := \{T \in \mathcal{T} : \eta_{T} \in [\eta,2\eta) \text{ and } \Delta_{T} \in [\Delta,2\Delta)\}. \end{displaymath}
Since $\mathcal{T} \subset \bigcup_{\eta,\Delta} \mathcal{T}(\eta,\Delta)$, there exist $\eta,\Delta$ such that, setting $\mathcal{T}_{0} := \mathcal{T}(\eta,\Delta)$,
\begin{displaymath} \sum_{T \in \mathcal{T}_{0}} |P(\mathfrak{r}) \cap T| \gtrapprox \sum_{T \in \mathcal{T}} |P(\mathfrak{r}) \cap T| \sim \mathfrak{r}|P(\mathfrak{r})|. \end{displaymath} 
Since $\eta_{T} \leq 2\eta$ for all $T \in \mathcal{T}_{1}$, the above combined with property (1) above (namely $|P_{T}| \geq \delta^{\eta_{T}}|P(\mathfrak{r}) \cap T|$) implies
\begin{displaymath} \mathcal{I}(P(\mathfrak{r}),\mathcal{T}_{0}) \stackrel{\mathrm{def.}}{=} \sum_{T \in \mathcal{T}_{0}} |P_{T}| \geq \delta^{2\eta} \sum_{T \in \mathcal{T}_{0}} |P(\mathfrak{r}) \cap T| \gtrapprox \delta^{2\eta}\mathfrak{r}|P(\mathfrak{r})|, \end{displaymath} 
as desired.
\end{proof}
We note that $\delta/\Delta$ is a small scale:

\begin{cor}\label{cor1} It holds $\delta/\Delta \lesssim \delta^{\epsilon/2}$.
\end{cor}

\begin{proof} Notice that since $\diam(P_{T}) \lesssim \Delta$ and $P_{T} \subset P_{r}(\mathcal{T})$ is $\delta$-separated for all $T \in \mathcal{T}_{0}$,
\begin{displaymath} \delta^{-\epsilon/2}|\mathcal{T}| \leq \delta^{2\eta - \epsilon}|\mathcal{T}| \stackrel{\eqref{form11}}{\leq} \delta^{2\eta}\mathfrak{r}|P(\mathfrak{r})| \stackrel{\mathrm{C. \,} \ref{c1}}{\lessapprox} \mathcal{I}(P(\mathfrak{r}),\mathcal{T}_{0}) = \sum_{T \in \mathcal{T}_{0}} |P_{T}| \lesssim |\mathcal{T}| (\Delta/\delta). \end{displaymath}
This gives the claim.
\end{proof} 

We proceed to the next refinement.

\begin{claim}\label{c2} There exist subsets $\mathcal{T}_{1} \subset \mathcal{T}_{0} \subset \mathcal{T}$, $P_{1} \subset P(\mathfrak{r})$ and a number $\ell \in 2^{\mathbb{N}} \cap \{1,\ldots,10\Delta/\delta\}$ with the following properties:
\begin{itemize}
\item[(i)] $P_{T} \cap P_{1} = P_{T}$ and $|P_{T}| \sim l$ for all $T \in \mathcal{T}_{1}$.
\item[(ii)] $\mathcal{I}(P_{1},\mathcal{T}_{1}) \gtrapprox \delta^{2\eta}\mathfrak{r}|P(\mathfrak{r})|$.
\item[(iii)] $|P_{1}| \gtrapprox \delta^{2\eta}|P(\mathfrak{r})|$.
\end{itemize}
\end{claim} 

\begin{proof} For $l \in 2^{\mathbb{N}} \cap \{1,\ldots,10\Delta/\delta\}$, write $\mathcal{T}_{0}(l) := \{T \in \mathcal{T}_{0} : |P_{T}| \in [l,2l)\}$ (evidently $|P_{T}| \leq 10\Delta/\delta$, since $P_{T}$ is a $\delta$-separated subset of $Q \cap T$). Then,,
\begin{displaymath} \delta^{2\eta}\mathfrak{r}|P(\mathfrak{r})| \stackrel{\mathrm{C.\, } \ref{c1}}{\lessapprox} \mathcal{I}(P(\mathfrak{r}),\mathcal{T}_{0}) \leq \sum_{l \in 2^{\mathbb{N}} \cap \{1,\ldots,10\Delta/\delta\}} \mathcal{I}(P(\mathfrak{r}),\mathcal{T}_{0}(l)). \end{displaymath} 
Thus, there exists $l \in 2^{\mathbb{N}} \cap \{1,\ldots,10\Delta/\delta\}$ such that $\mathcal{I}(P(\mathfrak{r}),\mathcal{T}_{0}(l)) \gtrapprox \delta^{2\eta}\mathfrak{r}|P(\mathfrak{r})|$. We now define
\begin{displaymath} \mathcal{T}_{1} := \mathcal{T}_{0}(l) \quad \text{and} \quad P_{1} := P(\mathfrak{r}) \cup \bigcup_{T \in \mathcal{T}_{1}} T. \end{displaymath}
Then evidently $P_{1} \cap P_{T} = P_{T}$ and $|P_{T}| \sim l$ for all $T \in \mathcal{T}_{1}$. Furthermore, 
\begin{displaymath} \mathcal{I}(P_{1},\mathcal{T}_{1}) = \sum_{p \in P_{1}} |\{T \in \mathcal{T}_{1} : p \in P_{T}\}| \stackrel{(\ast)}{=} \sum_{p \in P(\mathfrak{r})} |\{T \in \mathcal{T}_{0}(l) : p \in P_{T}\}| \gtrapprox \delta^{2\eta}\mathfrak{r}|P(\mathfrak{r})|, \end{displaymath}
where $(\ast)$ follows by observing that if $p \in P(\mathfrak{r})$ and $|\{T \in \mathcal{T}_{0}(l) : p \in P_{T}\}| \geq 1$, then $p \in P(\mathfrak{r}) \cap T$ for some $T \in \mathcal{T}_{0}(l) = \mathcal{T}_{1}$, and therefore $p \in P_{1}$.

The final claim $|P_{1}| \gtrapprox \delta^{2\eta}|P(\mathfrak{r})|$ follows from $\delta^{2\eta}\mathfrak{r}|P(\mathfrak{r})| \lessapprox \mathcal{I}(P_{1},\mathcal{T}_{1}) \leq 2\mathfrak{r}|P_{1}|$. \end{proof}

\begin{cor}\label{cor2} For $T \in \mathcal{T}_{1}$, the set $S_{Q}(P_{1}  \cap P_{T})$ is a Frostman $((\delta/\Delta),\epsilon/100,O_{\epsilon}(1))$-set. \end{cor}

\begin{proof} This follows immediately from Claim \ref{c2}(i), namely $P_{1} \cap P_{T} = P_{T}$ for $T \in \mathcal{T}_{1}$.   \end{proof} 

We make a final refinement at scale $\Delta$:
\begin{claim}\label{c4} There exists a family $\mathcal{Q} \subset \mathcal{D}_{\Delta}(P_{1})$ with the following properties:
\begin{itemize}
\item[(i)] $\delta^{2\eta}\mathfrak{r}|P_{1} \cap Q| \lessapprox \mathcal{I}(P_{1} \cap Q,\mathcal{T}_{1})$ for all $Q \in \mathcal{Q}$.
\item[(ii)] $|P(\mathfrak{r})| \lessapprox \delta^{-2\eta}\sum_{Q \in \mathcal{Q}} |P_{1} \cap Q|$. 
\end{itemize}
\end{claim} 

\begin{proof} Define $\mathcal{Q}$ via condition (i): 
\begin{displaymath} \mathcal{Q} := \{Q \in \mathcal{D}_{\Delta}(P_{1}) : \mathcal{I}(P_{1} \cap Q,\mathcal{T}_{1}) \gtrapprox \delta^{2\eta}\mathfrak{r}|P_{1} \cap Q|\}. \end{displaymath}
Then note that by Claim \ref{c2}(ii),
\begin{align*} \delta^{2\eta}\mathfrak{r}|P(\mathfrak{r})| \lessapprox \mathcal{I}(P_{1},\mathcal{T}_{1}) \leq \sum_{Q \in \mathcal{Q}} \mathcal{I}(P_{1} \cap Q,\mathcal{T}_{1}) + \sum_{Q \in \mathcal{D}_{\Delta}(P_{1}) \, \setminus \mathcal{Q}} \mathcal{I}(P_{1} \cap Q,\mathcal{T}_{1}). \end{align*}
If the ``$\gtrapprox$" constant in the definition of $\mathcal{Q}$ is chosen appropriately, the second sum cannot dominate the left hand side, so we infer from $P_{1} \subset P(\mathfrak{r})$ that
\begin{displaymath}  \delta^{2\eta}\mathfrak{r}|P(\mathfrak{r})| \lessapprox \sum_{Q \in \mathcal{Q}} \mathcal{I}(P_{1} \cap Q,\mathcal{T}_{1}) \lesssim \mathfrak{r} \sum_{Q \in \mathcal{Q}} |P_{1} \cap Q|. \end{displaymath}
This yields (ii). \end{proof} 

Next, we define
\begin{displaymath} \mathcal{T}_{1}(Q) := \{T \in \mathcal{T}_{1} : P_{T} \subset Q\}, \qquad Q \in \mathcal{Q}, \end{displaymath}
so $\mathcal{I}(P_{1} \cap Q,\mathcal{T}_{1}) \stackrel{\mathrm{def.}}{=} \sum_{T \in \mathcal{T}_{1}} |P_{1} \cap Q \cap P_{T}| = \sum_{T \in \mathcal{T}_{1}(Q)} |P_{1} \cap Q \cap P_{T}| \stackrel{\mathrm{def.}}{=} \mathcal{I}(P_{1} \cap Q,\mathcal{T}_{1}(Q))$. We then record that by Claim \ref{c4}(i), and since $|P_{T}| \lesssim |Q \cap T|_{\delta} \lesssim (\Delta/\delta)$,
\begin{displaymath} \delta^{2\eta}\mathfrak{r}|P_{1} \cap Q| \lessapprox \mathcal{I}(P_{1} \cap Q,\mathcal{T}_{1}(Q)) \leq \sum_{T \in \mathcal{T}_{1}(Q)} |P_{T}| \lesssim (\Delta/\delta) \cdot |\mathcal{T}_{1}(Q)|. \end{displaymath} 
This can be rearranged to $|P_{1} \cap Q| \lessapprox \delta^{-2\eta}(\Delta/\delta)  |\mathcal{T}_{1}(Q)|/\mathfrak{r}$. This "trivial bound" allows us to establish the special case of our main claim \eqref{form40} where $\mathfrak{r} \leq C_{\mathrm{KT}}C_{\mathrm{reg}}\Delta^{-s}$. Namely, by the previous estimate,
\begin{displaymath} |P_{1} \cap Q| \lessapprox (C_{\mathrm{KT}}C_{\mathrm{reg}})^{1/s}\frac{\delta^{-2\eta - 1}|\mathcal{T}_{1}(Q)|}{\mathfrak{r}^{(1 + s)/s}}, \qquad \mathfrak{r} \leq C_{\mathrm{KT}}C_{\mathrm{reg}}\Delta^{-s}. \end{displaymath}
Since $|\{Q \in \mathcal{Q} : P_{T} \subset Q\}| \leq 1$ for each $T \in \mathcal{T}$, it follows
\begin{align*} |P(\mathfrak{r})| & \stackrel{\mathrm{C.\,} \ref{c4}}{\lessapprox} \delta^{-2\eta} \sum_{Q \in \mathcal{Q}} |P_{1} \cap Q| \lessapprox (C_{\mathrm{KT}}C_{\mathrm{reg}})^{1/s}\frac{\delta^{-4\eta - 1}}{\mathfrak{r}^{(1 + s)/s}} \sum_{Q \in \mathcal{Q}} |\mathcal{T}_{1}(Q)|\\
& \leq (C_{\mathrm{KT}}C_{\mathrm{reg}})^{1/s}\frac{\delta^{-4\eta - 1}|\mathcal{T}|}{\mathfrak{r}^{(1 + s)/s}}, \quad \mathfrak{r} \leq C_{\mathrm{KT}}C_{\mathrm{reg}}\Delta^{-s}. \end{align*}
Since $\eta \leq \epsilon/10$, this completes the proof \eqref{form40} in the case $\mathfrak{r} \leq C_{\mathrm{KT}}C_{\mathrm{reg}}\Delta^{-s}$. So, in the sequel we may assume
\begin{equation}\label{form41} \mathfrak{r} \geq C_{\mathrm{KT}}C_{\mathrm{reg}}\Delta^{-s}. \end{equation}

To handle this case, we fix $Q \in \mathcal{Q}$, and define a family $\mathcal{R}(Q)$ of rectangles of dimensions $\delta \times \Delta$ (sometimes called ``tubelets" in the literature) intersecting $Q$. Such rectangles will be denoted with the letter $u$. We initially define $\mathcal{R}_{0}(Q)$ to be a maximal family of distinct such rectangles (where distinctness means $\mathrm{Leb}(u_{1} \cap u_{2}) \leq \tfrac{1}{2}\mathrm{Leb}(u_{1})$).  

For $u \in \mathcal{R}_{0}(Q)$ fixed, define the multiplicity number 
\begin{displaymath} m(u) := |\{T \in \mathcal{T}_{1}(Q) : T \prec u\}|, \end{displaymath}
where the notation $T \prec u$ means that $T \cap Q \subset Cu$. Here $C \geq 1$ is an absolute constant chosen so that each $T \in \mathcal{T}^{\delta}$ with $Q \cap T \neq \emptyset$ satisfies $T \prec u$ for at least one (and therefore $O(1)$) rectangles $u \in \mathcal{R}_{0}(Q)$. 

\begin{claim}\label{c3} For each $Q \in \mathcal{Q}$, there exists a family $\mathcal{R}(Q) \subset \mathcal{R}_{0}(Q)$ and a number $m_{Q} \in 2^{\mathbb{N}}$ with the following properties:
\begin{itemize}
\item[(i)] $m(u) \sim m_{Q}$ for all $u \in \mathcal{R}(Q)$.
\item[(ii)] $\delta^{2\eta}\mathfrak{r}|P_{1} \cap Q| \lessapprox \mathcal{I}(P_{1} \cap Q,\mathcal{T}_{1}(Q)) \lessapprox lm_{Q} |\mathcal{R}(Q)|$.
\end{itemize}
 \end{claim} 
 
 \begin{remark} Recall that $P_{T} \subset Q$ for all $T \in \mathcal{T}_{1}(Q)$. \end{remark} 
 
 \begin{proof}[Proof of Claim \ref{c3}] For $m \in 2^{\mathbb{N}}$, let $\mathcal{R}(Q,m) := \{u \in \mathcal{R}_{0}(Q) : m(u) \sim m\}$. Then, 
 \begin{align*} \mathcal{I}(P_{1} \cap Q,\mathcal{T}_{1}(Q)) & = \sum_{T \in \mathcal{T}_{1}(Q)} |P_{1} \cap P_{T}| \leq \sum_{m \in 2^{\mathbb{N}}} \sum_{u \in \mathcal{R}(Q,m)} \mathop{\sum_{T \in \mathcal{T}_{1}(Q)}}_{T \prec u} |P_{1} \cap P_{T}|\\
 & \stackrel{\mathrm{C.\,} \ref{c2}}{\lesssim} \sum_{m \in 2^{\mathbb{N}}} |\mathcal{R}(Q,m)| \cdot lm. \end{align*} 
 Consequently, there exists $m =: m_{Q}$ such that $\mathcal{R}(Q) := \mathcal{R}(Q,m)$ satisfies (ii). \end{proof} 
 
For $Q \in \mathcal{R}(Q)$ fixed, we record the observation 
 \begin{equation}\label{form9} l m_{Q} |\mathcal{R}(Q)| \stackrel{\mathrm{C.\,} \ref{c3}}{\gtrapprox} \delta^{2\eta}\mathfrak{r} |P_{1} \cap Q| \quad \Longrightarrow \quad l \gtrapprox \frac{\delta^{2\eta}\mathfrak{r}|P_{1} \cap Q|}{m_{Q} |\mathcal{R}(Q)|}. \end{equation}

 Recall that the ``original" tube-family $\mathcal{T}$ was assumed to be a Katz--Tao $(\delta,1,C_{\mathrm{KT}})$-set. The rectangles $\mathcal{R}(Q)$ form a Katz--Tao $(\delta/\Delta,1)$-set of $(\delta/\Delta)$-tubes in the following sense:
 
 \begin{claim}\label{c6} Let $Q \in \mathcal{Q}$, and let $S_{Q} \colon \R^{2} \to \R^{2}$ be the affine rescaling map taking $Q$ to $[0,1)^{2}$. Then the family $\{S_{Q}(u) : u \in \mathcal{R}(Q)\}$ is a Katz--Tao $((\delta/\Delta),1,O(C_{\mathrm{KT}}'))$-set (of distinct ordinary $(\delta/\Delta)$-tubes), where 
 \begin{displaymath} C_{\mathrm{KT}}' := C_{\mathrm{KT}}C_{\mathrm{reg}}m_{Q}^{-1}\Delta^{-s}, \end{displaymath}
Moreover, $|\mathcal{R}(Q)| \lesssim |\mathcal{T}_{1}(Q)|/m_{Q}$. \end{claim} 
 
 \begin{remark} We record that $m_{Q} \lesssim C_{\mathrm{reg}}\Delta^{-s}$, so $C_{\mathrm{KT}}' \gtrsim C_{\mathrm{KT}} \geq 1$. The reason is that $m(u) \lesssim C_{\mathrm{reg}}\Delta^{-s}$ uniformly. To see this, note that if $T \prec u$, then the slope of $T$ lies in some fixed interval $J = J(u)$ with $\diam(J) \lesssim \delta/\Delta$. For each slope $\sigma \in \delta \Z \cap J$ there are $\lesssim 1$ tubes $T \prec u$ with $\sigma(T) = \sigma$. On the other hand, since $\sigma(\mathcal{T})$ is $(\delta,s,C_{\mathrm{reg}})$-regular, we have $|\sigma(\mathcal{T}) \cap J| \lesssim C_{\mathrm{reg}}\Delta^{-s}$. \end{remark}
 
 \begin{proof}[Proof of Claim \ref{c6}] Let $\rho \in 2^{-\mathbb{N}} \cap [\delta,\Delta]$, and let $\mathbf{R}$ be a rectangle of dimensions $2\Delta \times \rho$. We need to show that 
 \begin{displaymath} |\{u \in \mathcal{R}(Q) : u \subset \mathbf{R}\}| \leq C_{\mathrm{KT}}'(\rho/\delta). \end{displaymath}
 To prove this, we first estimate
\begin{align} |\{u \in \mathcal{R}(Q) : u \subset \mathbf{R}\}| & \lesssim m_{Q}^{-1} \mathop{\sum_{u \in \mathcal{R}(Q)}}_{u \subset \mathbf{R}} |\{T \in \mathcal{T}_{1}(Q) : T \prec u\}| \notag\\
&\label{form29} \lesssim m_{Q}^{-1} |\{T \in \mathcal{T}_{1}(Q) : T \prec \mathbf{R}\}|. \end{align} 
Here $T \prec \mathbf{R}$ means the same as ``$T \prec u$ for at least one $u \in \mathcal{R}(Q)$ with $u \subset \mathbf{R}$". We record that the estimate $|\mathcal{R}(Q)| \lesssim |\mathcal{T}_{1}(Q)|/m_{Q}$ claimed as the last item in Claim \ref{c6} follows directly from the calculation above, just omitting the inclusion ``$u \subset \mathbf{R}$" everywhere.

Write $\mathcal{T}(\mathbf{R}) := \{T \in \mathcal{T}_{1}(Q) : T \prec \mathbf{R}\}$. To estimate $|\mathcal{T}(\mathbf{R})|$, let $\mathbf{T}_{1},\ldots,\mathbf{T}_{M} \in \mathcal{T}^{\rho}(\mathcal{T})$ be the list of dyadic $\rho$-tubes which contain at least one element of $\mathcal{T}(\mathbf{R}) \subset \mathcal{T}$, thus $\mathcal{T}(\mathbf{R}) \subset \bigcup_{j \leq M} (\mathcal{T} \cap \mathbf{T}_{j})$. By the Katz--Tao $(\delta,1,C_{\mathrm{KT}})$-set hypothesis,
\begin{displaymath} |\mathcal{T} \cap \mathbf{T}_{j}| \leq C_{\mathrm{KT}}(\rho/\delta), \qquad 1 \leq j \leq M. \end{displaymath}
On the other hand, the slope of each $\mathbf{T}_{j}$ lies on a fixed interval $J = J(\mathbf{R}) \subset [-1,1]$ of length $\lesssim \rho/\Delta$. This implies by the $(\delta,s,C_\mathrm{reg})$-regularity of $\sigma(\mathcal{T})$ that 
\begin{displaymath} M \lesssim |\sigma(\mathcal{T}) \cap J|_{\rho} \lesssim C_{\mathrm{reg}}\Delta^{-s}. \end{displaymath} 
Therefore, $|\{T \in \mathcal{T}_{1}(Q) : T \prec \mathbf{R}\}| \lesssim C_{\mathrm{KT}}C_{\mathrm{reg}}\Delta^{-s} \cdot (\rho/\delta)$, as claimed, and now the claim follows from \eqref{form29}. \end{proof}

 We apply Theorem \ref{t:WW} at scale $\delta/\Delta \lesssim \delta^{\epsilon/2}$ (Corollary \ref{cor1}) with constant $C = C_{\mathrm{KT}}'$ to
 \begin{itemize}
 \item the set $\bar{P} := S_{Q}(P_{1} \cap Q)$, 
 \item the family of $(\delta/\Delta)$-tubes $\mathbb{T} := \{S_{Q}(u) : u \in \mathcal{R}(Q)\}$,
 \item the subsets $\bar{P}_{S_{Q}(u)} := S_{Q}(P_{1} \cap P_{T}) \subset \bar{P} \cap S_{Q}(u)$,
 \end{itemize}
 where $T \in \mathcal{T}_{1}(Q)$ is an arbitrary tube satisfying $T \prec u$. Then $|\bar{P}_{S_{Q}(u)}| \geq |P_{1} \cap P_{T}| \sim l$ by Claim \ref{c2}. We also recall from Corollary \ref{cor2} that $\bar{P}_{S_{Q}(u)}$ is a $((\delta/\Delta),\eta,O_{\epsilon}(1))$-set with $\eta \leq \epsilon/10$. With this numerology, Theorem \ref{t:WW} implies\footnote{We hide the ``$\delta^{-\epsilon}$" constant from Theorem \ref{t:WW} in the $\lessapprox$ notation.}
 
 \begin{align*} |P_{1} \cap Q| \sim |\bar{P}|_{\delta/\Delta} \geq \Big| \bigcup_{\mathbf{T} \in \mathbb{T}} \bar{P}_{\mathbf{T}} \Big|_{\delta/\Delta} & \gtrapprox \frac{1}{C_{\mathrm{KT}}'}(\delta/\Delta)^{(1 + \eta)/2}l^{3/2}|\mathcal{R}(Q)|\\
 & \stackrel{\eqref{form9}}{\gtrapprox} \frac{\delta^{3\eta}}{C_{\mathrm{KT}}'}(\mathfrak{r}/m_{Q})^{3/2}|P_{1} \cap Q|^{3/2}(\delta/\Delta)^{(1 + \eta)/2}|\mathcal{R}(Q)|^{-1/2}. \end{align*} 
Taking also into account that $|\mathcal{R}(Q)| \lesssim |\mathcal{T}_{1}(Q)|/m_{Q}$, this can be rearranged into
 \begin{align*} |P_{1} \cap Q| & \lessapprox \delta^{-6\eta}(C_{\mathrm{KT}}')^{2}\frac{(\Delta/\delta)^{1 + \eta}|\mathcal{R}(Q)|}{(\mathfrak{r}/m_{Q})^{3}}\\
 & \stackrel{\mathrm{C.\,} \ref{c6}}{\lesssim}  \delta^{-6\eta}(C_{\mathrm{KT}}C_{\mathrm{reg}}\Delta^{-s})^{2} \cdot \frac{(\Delta/\delta)^{1 + \eta} |\mathcal{T}_{1}(Q)|}{\mathfrak{r}^{3}}. \end{align*}
 Here, using $\mathfrak{r} \geq C_{\mathrm{KT}}C_{\mathrm{reg}}\Delta^{-s}$ (recall \eqref{form41}), and $s \geq \tfrac{1}{2}$, it holds
  \begin{displaymath} \frac{(C_{\mathrm{KT}}C_{\mathrm{reg}}\Delta^{-s})^{2}\Delta}{\mathfrak{r}^{3}} = \frac{(C_{\mathrm{KT}}C_{\mathrm{reg}}\Delta^{-s})^{2 - 1/s} \cdot (C_{\mathrm{KT}}C_{\mathrm{reg}})^{1/s}}{\mathfrak{r}^{3}}\leq \frac{(C_{\mathrm{KT}}C_{\mathrm{reg}})^{1/s}}{\mathfrak{r}^{(s + 1)/s}}. 
  \end{displaymath}
  Consequently, 
  \begin{displaymath} |P_{1} \cap Q| \lessapprox \delta^{-7\eta} \frac{(C_{\mathrm{KT}}C_{\mathrm{reg}})^{1/s} \cdot \delta^{-1}|\mathcal{T}_{1}(Q)|}{\mathfrak{r}^{(s + 1)/s}}, \qquad Q \in \mathcal{Q}. \end{displaymath}
To finish the proof of \eqref{form40}, sum the preceding estimate over $Q \in \mathcal{Q}$, recall from Claim \ref{c4} that $|P(\mathfrak{r})| \lessapprox \delta^{-2\eta} \sum_{Q \in \mathcal{Q}} |P_{1} \cap Q|$, and note that $\sum_{Q \in \mathcal{Q}} |\mathcal{T}_{1}(Q)| \leq |\mathcal{T}|$. \end{proof}
\subsection{Sharpness of Theorem \ref{theorem: incidence bound}}\label{sec: sharp example}
The following example will demonstrate the sharpness of the factor $C_{\mathrm{reg}}^{1/s}$ in Theorem \ref{theorem: incidence bound}. 
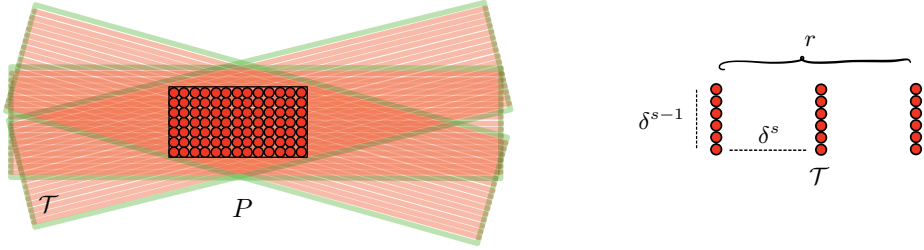
\begin{figure}[h!]
\begin{center}
\begin{overpic}[scale = 0.8]{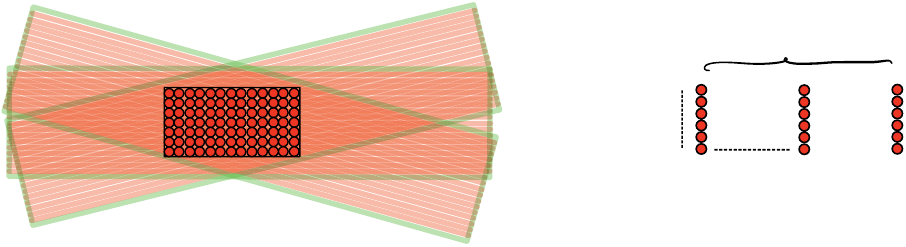}
\put(4,4){\small{$\mathcal{T}$}}
\put(25,3.5){$\small{P}$}
\put(82,11.5){\small{$\delta^{s}$}}
\put(69,13){\small{$\delta^{s - 1}$}}
\put(87,22){\small{$r$}}
\put(87.5,7){\small{$\mathcal{T}$}}
\end{overpic}
\caption{On the left: the set $P$ and the tubes $\mathcal{T}$. On the right: the dual squares of the tubes $\mathcal{T}$.}\label{fig1}
\end{center}
\end{figure}
\begin{example} For $C\geq 1$ and $s \in [\tfrac{1}{2},1)$ given, we will find (arbitrarily small) $\delta \in 2^{-\mathbb{N}}$, $1 \leq r \leq \delta^{-s}$, and a Katz-Tao $(\delta,1)$-set of $\delta$-tubes $\mathcal{T}$ such that $\sigma(\mathcal{T})$ is $(\delta,1,C)$-regular, and
\begin{equation}\label{form43} |\p_r(\mathcal{T})|_{\delta} \gtrsim C^{1/s} \cdot \frac{\delta^{-1}|\mathcal{T}|}{r^{(1 + s)/s}}. \end{equation}

Fix $\delta \in 2^{-\mathbb{N}}$ so small that $\delta^{-s(1 - s)} \geq C$. Fix also $r \in \{1,\ldots,\delta^{-s}\}$ such that $C \sim r^{1 - s}$. Let 
\begin{displaymath} P := [0,r^{-1}] \times [0,c\delta^{s}], \end{displaymath}
where $c > 0$ is a suitable absolute constant. Observe that $P$ is contained in many $\delta^{s}$-tubes. More precisely, there exists an arc $J \subset \mathbb{S}^{1}$ of directions with $\mathcal{H}^{1}(J) \sim r\delta^{s}$ such that for each $\theta \in J$, we have $P \subset \mathbf{T}$ for some $(1 \times \delta^{s})$-tube $\mathbf{T}_{\theta}$ whose longer side is parallel to $\theta$. 

Let $\Theta \subset J$ be a maximal $\sim \delta^{s}$-separated subset with $|\Theta| = r$, and let 
\begin{displaymath} \mathbb{T} := \{\mathbf{T}_{\theta} : \theta \in \Theta\}. \end{displaymath}
The tubes in $\mathbb{T}$ are depicted with green outlines in Figure \ref{fig1}. 

To define the announced Katz-Tao $(\delta,1)$-set of $\delta$-tubes $\mathcal{T}$, fix $\mathbf{T} \in \mathbb{T}$, and let $\mathcal{T}(\mathbf{T})$ be a maximal family of distinct parallel $\delta$-tubes contained in $\mathbf{T}$ (thus $\mathbf{T} \subset \bigcup_{T \in \mathcal{T}(\mathbf{T})} T$). Then, set
\begin{equation}\label{form44} \mathcal{T} := \bigcup_{\mathbf{T} \in \mathbb{T}} \mathcal{T}(\mathbf{T}) \quad \Longrightarrow \quad |\mathcal{T}| \sim r \cdot \delta^{s - 1}. \end{equation}
Note that every point $p \in P$ is contained in $|\Theta| = r$ many tubes in $\mathbb{T}$, and therefore $\geq r$ tubes in $\mathcal{T}$. In other words $P \subset \p_r(\mathcal{T})$, which gives $|\p_r(\mathcal{T})|_{\delta} \geq |P|_{\delta} \gtrsim \delta^{s - 2}/r$.

To understand the Katz-Tao and regularity properties of $\mathcal{T}$ and $\sigma(\mathcal{T})$, it is profitable to view the point-line dual of $\mathcal{T}$, shown on the right hand side of Figure \ref{fig1}. Each column of red discs in that picture represents the dual of some fixed $\mathcal{T}(\mathbf{T})$ -- consisting of parallel tubes. Evidently the dual of $\mathcal{T}$ (and therefore $\mathcal{T}$ itself) is a Katz-Tao $(\delta,1)$-set with a constant independent of $r \in [1,\delta^{-s}]$.

To bound the regularity constant of the slope set $\sigma(\mathcal{T})$ (viewed here as a subset of $\mathbb{S}^{1})$, note that $\sigma(\mathcal{T}) = \Theta$, because the tubes in $\mathcal{T}(\mathbf{T})$ are parallel to $\mathbf{T}$. Therefore, for $\delta \leq \rho \leq R \leq 1$,
\begin{displaymath} |\sigma(\mathcal{T}) \cap B(\theta,R)|_{\rho} = |\Theta \cap B(\theta,R)|_{\rho}.  \end{displaymath}
We claim that the right hand side is $\lesssim r^{1 - s}(R/\rho)^{s}$. Since $\diam(\Theta) \leq r\delta^{s}$, it suffices to consider $R \leq r\delta^{s}$, and since $\Theta$ is $\delta^{s}$-separated, it suffices to consider $\rho \geq \delta^{s}$. In this range $R/r \leq r$, and the claimed estimate follows from the trivial bound $|\Theta \cap B(\theta,R)|_{\rho} \lesssim (R/\rho) \leq r^{1 - s}(R/\rho)^{s}$.

Since $C \sim r^{1 - s}$, we have now shown that $\sigma(\mathcal{T})$ is $(\delta,s,O(C)))$-regular. Finally, the estimate \eqref{form43} follows from recalling that $|\p_r(\mathcal{T})|_{\delta} \gtrsim \delta^{s - 2}/r$, and observing that also
\begin{displaymath} |\mathcal{P}_{r}(\mathcal{T})| \gtrsim \delta^{s - 2}/r  = r^{(1 - s)/s} \cdot \frac{\delta^{-1} \cdot (r \cdot \delta^{s - 1})}{r^{(1 + s)/s}} \stackrel{\eqref{form44}}{\sim} C^{1/s} \cdot \frac{\delta^{-1}|\mathcal{T}|}{r^{(1 + s)/s}}.  \end{displaymath} 
\end{example}
\section{Nikod\'ym maximal estimate}
\label{sec: maximal estimate}
Let us recall the definition of quasi-Assouad dimension as introduced in \cite{qA_dimension}. Let $E\subseteq\R^d$ be a non-empty set. For $\gamma \in (0, 1)$, define
$$
h_E(\gamma):= \inf \left\{ \alpha : 
 \exists\,C > 0 \text{ such that,}\; 
 \text{for all}\, 0 < r<r^{1-\gamma}< R \;\text{and}\; x \in E,\;
|B(x, R) \cap E \big|_r \leqslant C \left( \frac{R}{r} \right)^\alpha 
\right\}.$$
The \textit{quasi-Assouad dimension} of $E$ is defined as the limit
\[
\dim_{\mathrm{qA}} E := \lim_{\gamma \to 0^+} h_E(\gamma).
\]
For the empty set, we define $\dim_{qA} \varnothing := 0$. We list some basic features of $\qad$ that will be used in \S\ref{sec: bochner--riesz}. 
\begin{lemma}
    The quasi-Assouad dimension satisfies the following properties.
    \begin{enumerate}
        \item Monotonicity: $\qad E\leq \qad F$ whenever $E\subseteq F$.
        \item Finite stability: $\qad(E\cup F)=\max\{\qad E,\qad F\}$.
        \item Bi-Lipschitz invariance: $\qad(f(E))=\qad(E)$ for any bi-Lipschitz map $f$.
        \item Stability under closure: $\qad(\overline{E})=\qad(E)$.
    \end{enumerate}
\end{lemma}
\begin{proof}
    See Table 3.1 in \cite[\S3.4.2]{fraser_book}.
\end{proof}
Let us first show the sharpness of the exponent $1+s$ in Theorem \ref{theorem: main}.
\begin{prop}\label{prop: sharpness}
    Let $\Theta\subseteq\mathbb{S}^1$ satisfy $\qad \Theta=s$, for $s\in[0,1]$. Suppose that for all $\epsilon>0$, there exists $C_\epsilon\geq 1$, such that
    \begin{equation}
        \label{eqn: sharpness nikodym bound}
        \|\N_{\Theta;\delta}f\|_{L^p(\R^2)}\leq C_\epsilon\delta^{-\epsilon}\|f\|_{L^p(\R^2)},
    \end{equation}
     for all $\delta\in(0,1)$. Then $p\geq 1+s$. 
\end{prop}
\begin{proof}
Let $\epsilon'>0$ be arbitrary. By the definition above, $h_\Theta(\gamma)\geq s-\epsilon'$ if $\gamma>0$ is chosen sufficiently small. By our hypothesis, 
for $\epsilon:=\gamma\epsilon'/p$, there exists $C_\epsilon\geq 1$ such that \eqref{eqn: sharpness nikodym bound} holds.
For this constant $C_\epsilon$, there exists a pair of scales $(\delta,\rho)$ satisfying $\rho>\delta^{1-\gamma}$, and a direction $\omega\in\Theta$, such that 
\begin{equation}
    \label{eqn: tube-count} |B(\omega,\rho)\cap\Theta|_\delta>C_\epsilon^p(\rho/\delta)^{s-2\epsilon'}.
\end{equation}
 Let $\mathbb{T}$ be a maximal family of $\delta$-tubes centred at the origin with $\delta$-separated directions lying in the set $B(\omega,\rho)\cap\Theta$. Then \eqref{eqn: tube-count} gives a lower bound on the size of this tube-family. Let $A:=\bigcup_{T\in\mathbb{T}}T$ and $R:=\bigcap_{T\in\mathbb{T}}T$. 
By geometric considerations, $R$ contains a rectangle of dimensions $\sim (\delta/\rho) \times \delta$, so $|R| \gtrsim \delta^{2}/\rho$. Also, $|A|\gtrsim\delta|\mathbb{T}|$, and by \eqref{eqn: tube-count}, we have $|\mathbb{T}|\gtrsim C_\epsilon^p(\rho/\delta)^{s-2\epsilon'}$.
Take $f=\chi_R$, and note that $$\N_{\Theta;\delta}f(x)\geq \frac{|R|}{|T|}\quad\text{for all}\quad x\in A \stackrel{\eqref{eqn: sharpness nikodym bound}}{\implies}|A|^{1/p}|R|\,\delta^{-1}\lesssim C_\epsilon\delta^{-\epsilon}|R|^{1/p}.$$ 
Using the lower bounds on $|A|$ and $|R|$ in the right-hand side above and simplifying, we get 
$$(\rho/\delta)^{\frac{s-2\epsilon'+1}{p}-1}\lesssim\delta^{-\epsilon}.$$
Recall that $\delta^{-\gamma}\leq(\rho/\delta)$, using which in the above, we arrive at $$1\lesssim(\rho/\delta)^{1-\frac{1+s-2\epsilon'}{p}-\frac{\epsilon}{\gamma}}=(\rho/\delta)^{1-\frac{1+s-3\epsilon'}{p}},$$ where the last step follows from our choice of $\epsilon$.
The above condition is true only if $p\geq 1+s-3\epsilon'$. Since $\epsilon'>0$ is arbitrary, this is equivalent to $p\geq 1+s$.
\end{proof}
\begin{remark}
    Proposition \ref{prop: sharpness} also demonstrates why the assumption $\overline{\dim}_\mathrm{B}(\Theta)\leq s$ is not sufficient to obtain improvements over C\'ordoba's $L^2$-bound (without restricting to radial functions). Indeed, by \cite[Proposition 1.6]{qA_dimension}, for any $\epsilon>0$, there exists $\Theta\subset\mathbb{S}^1$ satisfying $\overline{\dim}_\mathrm{B}(\Theta)=\epsilon$ but $\qad(\Theta)=1$, and consequently $p_\Theta\geq 2$.
\end{remark}
Next, we use Theorem \ref{theorem: incidence bound} to prove \eqref{eqn: main}.
\begin{prop}\label{prop: main}
    Suppose $\Theta\subseteq\mathbb{S}^1$ satisfies $\qad \Theta\leq s\in[1/2,1]$. Then for all $\epsilon>0$,
    $$\|\N_{\Theta;\delta}f\|_{L^{1+s}(\R^2)}\lesssim_\epsilon\delta^{-\epsilon}\|f\|_{L^{1+s}(\R^2)}.$$
\end{prop}
\begin{proof}
Fix $\epsilon>0$ and $\delta \in 2^{-\mathbb{N}}$. The hypothesis $\qad \Theta \leq s$ implies that $\Theta$ is $(\delta,s)$-regular:
\begin{claim}\label{c7} $\Theta$ is $(\delta,s,C_{\mathrm{reg}})$-regular for $C_{\mathrm{reg}} \lesssim_{\epsilon} \delta^{-\epsilon}$. \end{claim}

\begin{proof} Choose $\gamma>0$ small enough that $\gamma(1-s)<\epsilon$. By the definition of $h_{\Theta}(\gamma)$, for all scales $\delta\leq r\leq r^{1-\gamma}\leq R\leq 1$, there exists a constant $C_\epsilon\geq 1$, such that $$|\Theta\cap I|_r\leq C_\epsilon(R/r)^{h_{\Theta}(\gamma)+\epsilon}\leq C_\epsilon\delta^{-\epsilon}(R/r)^s,$$ for all $I\in\mathcal{D}_R(\R)$. On the other hand, when $\delta\leq r\leq R\leq r^{1-\gamma}\leq 1$, we have the trivial bound $$|\Theta\cap I|_r\leq R/r=(R/r)^{1-s}(R/r)^s\leq r^{-\gamma(1-s)}(R/r)^s\leq\delta^{-\epsilon}(R/r)^s.$$ 
This completes the proof of the claim. \end{proof}

We start the proof of Proposition \ref{prop: main} in earnest. Arguing as in Remark \ref{rem1} (and by translation invariance), it suffices to prove the following local estimate:
\begin{equation}
    \label{eqn: local estimate}
    \|\N_{\Theta;\delta} f\|_{L^{1+s}([-\tfrac{1}{2},\tfrac{1}{2}]^2)}\lesssim_\epsilon\delta^{-\epsilon}\|f\|_{L^{1+s}(\R^2)}.
\end{equation}
Next, we use local constancy of $\N_{\Theta;\delta}$ to discretise the problem. Abbreviate $\mathcal{D} := \mathcal{D}_{\delta}$. Let $c_q$ denote the centre of $q\in\mathcal{D}$. Any $\delta$-tube $T$ centred at $x\in q$ is contained in a $4\delta$-tube $T'$ centred at $c_q$, with the same direction as $T$. Thus, for all $x\in q$, we have 
$$\N_{\Theta;\delta} f(x)\lesssim\sup_{T\ni c_q}\frac{1}{|T|}\int_{T}|f(y)|\,\dd{y}\lesssim\frac{1}{|T_q|}\int_{T_q}|f(y)|\,\dd{y},$$
where the supremum is taken over all $4\delta$-tubes centred at $c_q$, and essentially realised by the tube $T_q$. Then, $$\|\N_{\Theta;\delta} f\|_{L^{1+s}([-\tfrac{1}{2},\tfrac{1}{2}]^{2})}^{1+s}=\sum_{q\in\mathcal{D}}\|\N_{\Theta;\delta} f\|_{L^{1+s}(q)}^{1+s}\lesssim\delta^{1-s}\sum_{q\in\mathcal{D}}\bigg(\int_{T_q}|f(y)|\,\dd{y}\bigg)^{1+s}=\delta^{1-s}\sum_{q\in\mathcal{D}}a_q\int_{T_q}|f(y)|\,\dd{y},$$ 
for some sequence $a:=(a_q)_{q\in\mathcal{D}}$ with $\|a\|_{\ell^{1+1/s}(\mathcal{D})}=1$. By H\"older's inequality, it follows that 
$$\|\N_{\Theta;\delta} f\|_{L^{1+s}([-\tfrac{1}{2},\tfrac{1}{2}]^{2})}\lesssim\delta^{\frac{1-s}{1+s}}\|\sum_{q\in\mathcal{D}}a_q\chi_{T_q}\|_{L^{1+1/s}(\R^2)}\|f\|_{L^{1+s}(\R^2)}.$$
Thus, the goal \eqref{eqn: local estimate} is reduced to proving the following bound
\begin{equation*}
    \label{eqn: nikodym dual with coefficients}
    \|\sum_{q\in\mathcal{D}}a_q\mathbf{1}_{T_q}\|_{L^{1+1/s}(\R^2)}\lesssim_\epsilon\delta^{-\frac{1-s}{1+s}-\epsilon}\|a\|_{\ell^{1+1/s}(\mathcal{D})},
\end{equation*}
for all sequences $a=(a_q)_{q\in\mathcal{D}}$.
By dyadically pigeonholing in the sizes of the coefficients, we can further reduce matters to the case where $a_q=1$ for all $q$. That is, 
\begin{equation}
    \label{eqn: nikodym dual}
    \Big\| \sum_{q\in\mathcal{D}}\mathbf{1}_{T_q}\Big\|_{L^{1+1/s}(\R^2)}\lesssim_\epsilon\delta^{-1-\epsilon}.
\end{equation}
To prove \eqref{eqn: nikodym dual}, we may assume that that $\Theta \subset B(e_{1},\tfrac{1}{10})$, that is, the directions of the tubes $T_{q}$ are all nearly horizontal. This can be achieved by partitioning $\{T_{q}\}_{q \in \mathcal{D}}$ into boundedly many sub-families, and applying rotations. The point of making this reduction is to ensure that of the "ordinary" $4\delta$-tubes $T_{q}$ appearing in \eqref{eqn: nikodym dual} may be covered by $C \lesssim 1$ dyadic $\delta$-tubes $T_{q}^{1},\ldots,T_{q}^{C}$ whose slopes are parallel to the slope of $T_{q}$, up to an error $\lesssim \delta$, and such that $\dist(q,T_{q}^{j}) \lesssim \delta$ for all $1 \leq j \leq C$. We leave the details to the reader. Thus, in order to prove \eqref{eqn: nikodym dual}, it suffices to verify the following variant stated in terms of dyadic $\delta$-tubes:
\begin{claim}\label{c5} Assume that for every $q \in \mathcal{D}$ we are given a single tube $T_{q} \in \mathcal{T}^{\delta}$ such that $\dist(q,T_{q}) \leq A\delta$. Let $\mathcal{T} := \bigcup_{q \in \mathcal{D}} T_{q}$, and assume that $\sigma(\mathcal{T}) \subset \delta \Z \cap [-1,1]$ is $(\delta,s,C_{\mathrm{reg}})$-regular. Then,
\begin{equation}\label{form57} \Big\| \sum_{q\in\mathcal{D}}\mathbf{1}_{T_q}\Big\|_{L^{1+1/s}([0,1]^{2})}\lesssim_{A,\epsilon} C_{\mathrm{reg}}^{1/(1 + s)}\delta^{-1-\epsilon}, \qquad \epsilon > 0. \end{equation}  \end{claim}
The plan is to apply Theorem \ref{theorem: incidence bound} to the family $\mathcal{T}$ in Claim \ref{c5}, but a small problem is that the tubes $T_{q}$, $q \in \mathcal{D}$, may not be distinct. To quantify this, write $M(T) := |\{q \in \mathcal{D} : T = T_{q}\}|$ for $T \in \mathcal{T}$. Noting that $M(T) \in \{1,\ldots,O(\delta^{-1})\}$, we may decompose
\begin{displaymath} \sum_{q \in \mathcal{D}} \mathbf{1}_{T_{q}} \leq  \sum_{M = 1}^{O(\delta^{-1})} 2M \sum_{T \in \mathcal{T}_{M}} \mathbf{1}_{T}, \end{displaymath}
where the $M$-sum only runs over dyadic integers, and $\mathcal{T}_{M} := \{T \in \mathcal{T} : M(T) \in [M,2M]\}$. We see that \eqref{form57} follows from
\begin{equation}\label{form58} \Big\| \sum_{T \in \mathcal{T}_{M}} \mathbf{1}_{T} \Big\|_{L^{1 + 1/s}([0,1]^{2})} \lesssim_{A,\epsilon} C_{\mathrm{reg}}^{1/(1 + s)}\delta^{-1 - \epsilon}/M, \qquad M \in \{1,\ldots,O(\delta^{-1})\}. \end{equation}   
The constant $M$ shows up in the Katz-Tao constant of $\mathcal{T}_{M}$, as follows:

\begin{claim}\label{c8} $\mathcal{T}_{M}$ is a Katz-Tao $(\delta,1,C_{\mathrm{KT}})$-set with $C_{\mathrm{KT}} \lesssim A\delta^{-1}/M$. In particular $|\mathcal{T}_{M}| \lesssim A\delta^{-2}/M$. \end{claim} 

\begin{proof} Let $\delta \leq \Delta \leq 1$, and $\mathbf{T} \in \mathcal{T}^{\Delta}$. Note that if $T \in \mathcal{T}_{M}$ is contained in $\mathbf{T}$, then by the hypothesis of Claim \ref{c5} all the squares $q \in \mathcal{D}$ with $T = T_{q}$ are contained in the $(A\delta)$-neighbourhood of $\mathbf{T}$, denoted $\mathbf{T}^{(A\delta)}$. Consequently, using also the definition of $\mathcal{T}_{M}$,
\begin{displaymath} |\mathcal{T}_{M} \cap \mathbf{T}| \leq \sum_{T \in \mathcal{T}_{M} \cap \mathbf{T}} \frac{|\{q \in \mathcal{D} \cap \mathbf{T}^{(A\delta)} : T = T_{q}\}|}{M} \leq \frac{1}{M} \sum_{q \in \mathcal{D} \cap \mathbf{T}^{(A\delta)}} |\{T \in \mathcal{T} : T = T_{q}\}|. \end{displaymath} 
But here $|\{T \in \mathcal{T} : T = T_{q}\}| = 1$ for all $q \in \mathcal{D}$ fixed, and $|\mathcal{D} \cap \mathbf{T}^{(A\delta)}| \lesssim A\Delta/\delta^{2}$. So, we may infer that $|\mathcal{T} \cap \mathbf{T}| \lesssim (A\delta^{-1}/M) \cdot (\Delta/\delta)$, as claimed. \end{proof} 

Now we may complete the proof of Claim \ref{c5} by appealing to Theorem \ref{theorem: incidence bound}. As in the theorem, we write $P_{r}(\mathcal{T}_{M}) := \{p \in [0,1]^{2} : |\{T \in \mathcal{T}_{M} : p \in T\}| \geq r\}$ for $r \geq 1$, and then decompose
\begin{equation}
    \label{eqn: tube sum bound}
    \|\sum_{T\in \mathcal{T}_{M}} \mathbf{1}_{T}\|_{L^{1+1/s}([0,1]^{2})}^{1+1/s} \lesssim  \sum_{1\leq r\leq\delta^{-2}} r^{1 + 1/s}\mathrm{Leb}(P_{r}(\mathcal{T}_{M})), \end{equation} 
    where the $r$-sum only runs over dyadic integers. Evidently $\mathrm{Leb}(P_{r}(\mathcal{T}_{M})) \lesssim \delta^{2} |P_{r}(\mathcal{T}_{M})|_{\delta}$. Plugging in the bound for the Katz-Tao constant provided by Claim \ref{c8}, Theorem \ref{theorem: incidence bound} yields
    \begin{displaymath} |P_{r}(\mathcal{T}_{M})|_{\delta}  \lesssim_{\epsilon} \delta^{-\epsilon}(C_{\mathrm{KT}}C_{\mathrm{reg}})^{1/s}\frac{\delta^{-1}|\T|}{r^{1 + 1/s}} \lesssim_{A,\epsilon} C_{\mathrm{reg}}^{1/s}\delta^{-\epsilon}\Big( \frac{\delta^{-1}}{M} \Big)^{1/s}\frac{\delta^{-3}/M}{r^{1 + 1/s}} = C_{\mathrm{reg}}^{1/s}\frac{\delta^{-2 - \epsilon} \cdot \delta^{-1 - 1/s}}{(Mr)^{1 + 1/s}}. \end{displaymath} 
    Inserting this bound back into \eqref{eqn: tube sum bound} proves \eqref{form58}, and therefore Claim \ref{c5}. \end{proof}

\subsection{Examples}
The next example displays a set $\Theta \subset \mathbb{S}^{1}$ with $\dim_{\mathrm{qA}} \Theta \leq \tfrac{1}{3}$, but such that the associated operators $\mathcal{N}_{\Theta;\delta}$ are not essentially bounded on $L^{q}$ for any $q < 3/2$. In particular, they are not essentially bounded on $L^{q}$ for $q = 1 + \dim_{\mathrm{qA}} \Theta$. 

\begin{prop}\label{ex3} There exists a set $\Theta \subset \mathbb{S}^{1}$ with $\dim_{\mathrm{qA}} \Theta \leq \tfrac{1}{3}$, and two vanishing sequences $\{\delta_{n}\}$ and $\{\epsilon_{n}\}$ such that 
\begin{equation}\label{form55} \|\mathcal{N}_{\Theta;\delta_{n}}\|_{L^{q} \to L^{q}} \geq \delta^{\tfrac{2}{3} - \tfrac{1}{q} + \epsilon_{n}}, \qquad n \in \mathbb{N}. \end{equation}
\end{prop}

The main work behind Proposition \ref{ex3} is contained in the following ``single-scale" version:

\begin{prop}\label{ex2} Fix $\epsilon > 0$, and an arc $I \subset \mathbb{S}^{1}$. Then, for $\delta > 0$ arbitrarily small there exists a $\delta$-separated set $\Theta \subset I$ satisfying
\begin{equation}\label{form49} |\Theta \cap B(\theta,R)|_{r} \leq \left(\tfrac{R}{r} \right)^{1/3 + \epsilon}, \qquad \theta \in \mathbb{S}^{1}, \,  \delta \leq r \leq r^{1 - \epsilon} \leq R, \end{equation} 
and such that $\|\mathcal{N}_{\Theta;\delta}\|_{L^{q} \to L^{q}} \geq \delta^{\tfrac{2}{3} - \tfrac{1}{q} + \epsilon}$ for all $q \geq 1$.
\end{prop} 
Let us complete the proof of Proposition \ref{ex3} based on Proposition \ref{ex2}.

\begin{proof}[Proof of Proposition \ref{ex3} assuming Proposition \ref{ex2}] We construct $\Theta$ recursively. Set 
\begin{displaymath} \Theta_{0} := \mathbb{S}^{1} \quad \text{and} \quad \delta_{0} := \mathcal{H}^{1}(\mathbb{S}^{1}), \end{displaymath}
and let $\{\epsilon_{n}\}$ be an arbitrary sequence with $\epsilon_{n} \searrow 0$. Assume that $\Theta_{n}$ has already been constructed for some $n \geq 0$, and can be written as a union of finitely many $\delta_{n}$-separated points, and a single arc of length $\delta_{n}$, denoted $I_{n}$. 

To construct $\Theta_{n + 1}$, apply Proposition \ref{ex2} with parameters $\epsilon_{n + 1}$ and $I_{n}$. This yields a scale $\delta_{n + 1} < \delta_{n}$, and a $\delta_{n + 1}$-separated set $\Theta_{n + 1}(I_{n}) \subset I_{n}$. Replace exactly one of the points $\theta \in \Theta_{n + 1}(I_{n})$ by an arc $I_{n + 1}$ centred at $\theta$, of length $\delta_{n + 1}$. Define $\Theta_{n + 1}$ as the union of the "old" points in $\Theta_{n}$, the points in $\Theta_{n + 1}(I_{n})$, and the new arc $I_{n + 1}$. 

Eventually, the sets $\Theta_{n}$ converge (in the Hausdorff metric) to a countable set $\Theta \subset \mathbb{S}^{1}$. We claim that $\dim_{\mathrm{qA}} \Theta \leq \tfrac{1}{3}$. Equivalently, for every $\epsilon > 0$ there exists a constant $C_{\epsilon} \geq 1$ such that
\begin{equation}\label{form54} |\Theta \cap B(\theta,R)|_{r} \leq C_{\epsilon}(R/r)^{1/3 + \epsilon}, \qquad \theta \in \mathbb{S}^{1}, \, 0 < r \leq r^{1 - \epsilon} \leq R. \end{equation}
 To prove \eqref{form54}, fix $\epsilon > 0$ and $0 < r \leq r^{1 - \epsilon} \leq R$. Let $n \in \mathbb{N}$ be the first index such that $\epsilon_{n + 1} \leq \tfrac{1}{2}\epsilon$; thus $\delta_{n} \sim_{\epsilon} 1$. Consider first the case $R \leq \delta_{n}$. In this case we claim that \eqref{form54} holds with $C_{\epsilon} = 1$:
\begin{equation}\label{form53} |\Theta \cap B(\theta,R)|_{r} \leq (R/r)^{1/3 + \epsilon}, \qquad \theta \in \mathbb{S}^{1}, \, r \leq r^{1 - \epsilon} \leq R \leq \delta_{n}. \end{equation}
To prove \eqref{form53}, consider first the sub-case where $\delta_{m + 1} \leq r \leq r^{1 - \epsilon} \leq R \leq \delta_{m}$ for some $m \geq n$. Fix $\theta \in \mathbb{S}^{1}$. Note that $\Theta \, \setminus \, I_{m} \subset \Theta_{m}$ is $\delta_{m}$-separated, so $|B(\theta,R) \cap (\Theta \, \setminus \, I_{m})| \leq 1$. On the other hand, 
\begin{displaymath} |(\Theta \cap I_{m}) \cap B(\theta,R)|_{r} \lesssim |\Theta_{m + 1}(I) \cap B(\theta,R)|_{r} \leq \left(\tfrac{R}{r} \right)^{1/3 + \epsilon_{m + 1}}. \end{displaymath} 
by the definition of $\Theta_{m + 1}(I)$ via Proposition \ref{ex2}. These observations show that $|\Theta \cap B(\theta,R)|_{r} \lesssim (R/r)^{1/3 + \epsilon_{m + 1}}$, and in particular
\begin{displaymath} |\Theta \cap B(\theta,\delta_{m})|_{\delta_{m + 1}} \lesssim (\delta_{m}/\delta_{m + 1})^{1/3 + \epsilon_{m + 1}}. \end{displaymath}
Since $\epsilon_{m + 1} \leq \epsilon_{n + 1} \leq \tfrac{1}{2}\epsilon$, and all the ratios $(R/r)$ and $(\delta_{m}/\delta_{m + 1})$ appearing above can be assumed to be large, we have
\begin{equation}\label{form52} |\Theta \cap B(\theta,R)|_{r} \leq (R/r)^{1/3 + \epsilon} \quad \text{and} \quad |\Theta \cap B(\theta,\delta_{m})|_{\delta_{m + 1}} \leq (\delta_{m}/\delta_{m + 1})^{1/3 + \epsilon}. \end{equation} 

Let us next consider the more general case where $R \leq \delta_{n}$, but now 
\begin{displaymath} \delta_{m + 1} \leq r \leq \delta_{m} \leq \delta_{k + 1} \leq R \leq \delta_{k} \end{displaymath}
for some $k < m$. We note that if $A \subset \mathbb{S}^{1}$ and $r < r' < R$, then $|A \cap B(\theta,R)|_{r} \leq |A \cap B(\theta,R)|_{r'}|A \cap B(\theta',r')|_{r}$ for some $\theta' \in \mathbb{S}^{1}$, namely the point which maximises $\theta' \mapsto |A \cap B(\theta',r')|_{r}$. Iterating this observation, we may estimate
\begin{displaymath} |\Theta \cap B(\theta,R)|_{r} \leq |\Theta \cap B(\theta,R)|_{\delta_{k + 1}} \big( \prod_{j = k + 1}^{m - 1} |\Theta \cap B(\theta_{j},\delta_{j})|_{\delta_{j + 1}} \Big) |\Theta \cap B(\theta_{m},\delta_{m})|_{r}\end{displaymath}
for some points $\theta_{k + 1},\ldots,\theta_{m} \in \mathbb{S}^{1}$. To each of the factors above we now apply \eqref{form52} to conclude \eqref{form53}. 

Finally, to prove \eqref{form54} in full generality, assume that $R \geq \delta_{n}$, but $r \leq r^{1 - \epsilon} \leq \delta_{n}$. Then,
\begin{displaymath} |\Theta \cap B(\theta,R)|_{r} \leq |\Theta \cap B(\theta,R)|_{\delta_{n}}|\Theta \cap B(\theta',\delta_{n})|_{r} \lesssim_{\epsilon} (\delta_{n}/r)^{1/3 + \epsilon} \leq (R/r)^{1/3 + \epsilon}. \end{displaymath} 
This only leaves the case $R \geq \delta_{n}$ and $r \geq \delta_{n}^{1/(1 - \epsilon)}$, where the trivial estimate $|\Theta \cap B(\theta,R)|_{r} \lesssim_{\epsilon} 1$ suffices to complete the proof of \eqref{form54}.

We have now verified that $\dim_{\mathrm{qA}} \Theta \leq \tfrac{1}{3}$. It remains to show that \eqref{form55} holds. This follows from the fact that $\Theta_{n + 1}(I_{n})$ is contained in the $\delta_{n + 1}$-neighbourhood of $\Theta$ (if desired, one could also slightly modify the construction to ensure that $\Theta_{n + 1}(I_{n}) \subset \Theta$ for all $n \geq 0$). Therefore,
\begin{displaymath} \|\mathcal{N}_{\Theta;\delta_{n + 1}}\|_{L^{q} \to L^{q}} \gtrsim \|\mathcal{N}_{\Theta(I_{n});\delta_{n + 1}}\|_{L^{q} \to L^{q}} \geq \delta_{n + 1}^{\tfrac{2}{3} - \tfrac{1}{q} + \epsilon_{n + 1}}, \qquad n \in \mathbb{N}, \, q \geq 1,  \end{displaymath} 
by the definition of $\Theta(I_{n})$ (that is, Proposition \ref{ex2}). \end{proof} 

The proof of Proposition \ref{ex2} is based on the following special case of \cite[Theorem 1.9]{2024arXiv241104528O}. 

\begin{example}\label{ex1} Fix $t > \tfrac{2}{3}$, and an interval $I \subset [0,1)$. Then there exists $C = C(I,t) > 0$ such that the following holds for $\delta > 0$ small enough. There exists a $\delta$-separated Frostman $(\delta,1,C)$-set $P \subset [0,1]^{2}$, and a $\delta$-separated $(\delta,\tfrac{1}{3},(\log(1/\delta))^{C})$-regular set $\Theta_{0} \subset I$ with cardinality $|\Theta_{0}| \geq (\log (1/\delta))^{-C}\delta^{-1/3}$ such that 
\begin{equation}\label{form47} |\pi_{\theta}(P)|_{\delta} \leq \delta^{-t}, \qquad \theta \in \Theta_{0}. \end{equation}
Here $\pi_{\theta}(x,y) := x + \theta y$ for $(x,y) \in \R^{2}$ and $\theta \in \R$.
\end{example} 

\begin{remark} Only the case $I = [0,1)$ has been stated in \cite[Theorem 1.9]{2024arXiv241104528O}, but the general case can be reduced to the case $I = [0,1)$, as follows. Write $I = [a,a + r)$, and consider the maps $\iota(\theta) := a + r\theta$ (which send $[0,1)$ to $I$) and $A(x,y) := (rx - ay,y)$. Then,
\begin{displaymath} (\pi_{\iota(\theta)} \circ A)(x,y) = (rx - ay) + (a + r\theta)y = r\pi_{\theta}(x,y) \end{displaymath}
for all $(x,y) \in \R^{2}$. Let $P \subset [0,1]^{2}$ and $\Theta_{0} \subset [0,1)$ be the sets provided by \cite[Theorem 1.9]{2024arXiv241104528O} in the special case $I = [0,1)$. Then the previous equation shows that 
\begin{displaymath} \pi_{\iota(\theta)}A(P) = r\pi_{\theta}(P), \qquad \theta \in \Theta_{0}. \end{displaymath}
Write $P' := A(P)$ and $\Theta_{0}' := \iota(\Theta_{0}) \subset I$. Now the previous equation implies, along with \eqref{form47}, that $|\pi_{\theta'}(P')|_{\delta} \leq |\pi_{\theta}(P)|_{\delta} \leq \delta^{-t}$ for all $\theta \in \Theta_{0}'$. One needs a few further minor edits to make sure that $P' \subset [0,1]^{2}$, and $\Theta_{0}'$ is $\delta$-separated, but we leave these to the reader (in these details, it is important that "$r$" can be viewed as a constant). \end{remark} 

We will also need a special case of the \emph{Furstenberg set theorem} \cite[Theorem 5.7]{2023arXiv230110199O}:
\begin{theorem}\label{t:FurstenbergSets} For every $\epsilon > 0$ there exists $\zeta > 0$ such that the following holds for all $\delta > 0$ sufficiently small. Let $P \subset [0,1]^{2}$ be a $(\delta,1,\delta^{-\zeta})$-regular set with $|P|_{\delta} \geq \delta^{\zeta - 1}$. For each $p \in P$, let $\mathcal{T}(p)$ be a Frostman $(\delta,\tfrac{1}{3},\delta^{-\zeta})$-set of ordinary $\delta$-tubes containing $p$. Then the union $\bigcup_{T \in \mathcal{T}} \mathcal{T}(p)$ contains a Frostman $(\delta,1,\delta^{-\epsilon})$-set (of ordinary $\delta$-tubes). \end{theorem}

\begin{remark} The conclusion of \cite[Theorem 5.7]{2023arXiv230110199O} is stated in the slightly weaker form that $\mathcal{T}$ contains a $\delta$-separated set of cardinality $\geq \delta^{\epsilon - 1}$. However, this conclusion applied at various different scales formally implies Theorem \ref{t:FurstenbergSets}. A similar argument appears in the proof \cite[Theorem 4.9]{MR4722034}.   \end{remark} 

We are now equipped to give the details of Example \ref{ex2}. Fix $\epsilon > 0$, and let 
\begin{displaymath} t  := 2/3 + \epsilon/2. \end{displaymath}
The direction set $\Theta$ will be a subset of the family $\Theta_{0} \subset I$ given by Example \ref{ex1} (we identify the given arc $I \subset \mathbb{S}^{1}$ with a subinterval of $[0,1)$ in order to apply Example \ref{ex1}). Fix $\zeta = \zeta(\epsilon) > 0$ to be determined later, and let $J \subset I$ be an arc of length $R_{0} := \delta^{\zeta}$ which contains maximally many points of $\Theta_{0}$. Write $\Theta := \Theta_{0} \cap J$. Evidently
\begin{displaymath} |\Theta| \geq \delta^{\zeta}|\Theta| \gtrsim (\log(1/\delta))^{-C}\delta^{\zeta - 1/3}. \end{displaymath}
In particular $\Theta$ is a Frostman $(\delta,\tfrac{1}{3},\delta^{-2\zeta})$-set. We check that $\Theta$ satisfies \eqref{form49}. Fix $r \leq r^{1 - \epsilon} \leq R$. We may assume that $R \leq R_{0} \sim \diam(\Theta)$. Then, by the $(\delta,\tfrac{1}{3},(\log(1/\delta))^{C})$-regularity of $\Theta_{0}$,
\begin{displaymath} |\Theta \cap B(\theta,R)|_{r} \leq (\log(1/\delta))^{C}\left(\tfrac{R}{r} \right)^{1/3} \leq \left(\tfrac{R}{r} \right)^{1/3 + \epsilon}, \end{displaymath} 
taking into account that $(R/r)^{\epsilon} \geq r^{-\epsilon} \geq R_{0}^{-\epsilon} \geq \delta^{-\epsilon \zeta} \geq (\log(1/\delta))^{C}$.

Let $P \subset [0,1]^{2}$ be the Frostman $(\delta,1,O_{I}(1))$-set provided by Example \ref{ex1}. Below, we denote the $\delta$-neighbourhood of a set $A \subset \R^{d}$ by $A_{\delta}$. The plan is to construct a Borel set $G \subset \R^{2}$ satisfying
\begin{equation}\label{form48} \mathrm{Leb}(G) \geq \delta^{\epsilon/4} \end{equation}
with the property 
\begin{equation}\label{form50} \mathcal{N}_{\Theta}(\mathbf{1}_{P_{\delta}})(x) \gtrsim_{I} \delta^{t}, \qquad x \in G. \end{equation}
A combination of \eqref{form48}-\eqref{form49} allows us to find a lower bound for $B_{q} := \|\mathcal{N}_{\Theta}\|_{L^{q} \to L^{q}}$:
\begin{displaymath} \delta^{t + \epsilon/4} \lesssim_{I} \|\mathcal{N}_{\Theta}(\mathbf{1}_{P_{\delta}})\|_{L^{q}(G)} \leq B_{q}\|\mathbf{1}_{P_{\delta}}\|_{L^{q}} \sim B_{q}\delta^{1/q}. \end{displaymath}
This yields $B_{q} \gtrsim_{I} \delta^{t - 1/q + \epsilon/2} = \delta^{2/3 - 1/q + 3\epsilon/4}$, and proposition follows.

It remains to construct $G$, and to this end we first define a family $\mathcal{T}$ of ordinary $\delta$-tubes. Fix $\theta \in \Theta$, and let $\mathcal{T}_{\theta,0}$ be a minimal cover of $P_{\delta}$ by ordinary $\delta$-tubes parallel to $\pi_{\theta}^{-1}\{0\}$. Then $|\mathcal{T}_{\theta,0}| \leq C\delta^{-t}$ by \eqref{form47}. Let $\mathcal{T}_{\theta}  := \{T \in \mathcal{T}_{\theta,0} : \mathrm{Leb}(T \cap P_{\delta}) \geq c\delta^{t + 1}\}$ for a suitable constant $c = c_{I} > 0$ to be determined in a moment. Then,
\begin{displaymath} \delta \sim_{I} \mathrm{Leb}(P_{\delta}) \leq \sum_{T \in \mathcal{T}_{\theta}} \mathrm{Leb}(T \cap P_{\delta}) + \sum_{T \in \mathcal{T}_{\theta,0} \, \setminus \, \mathcal{T}_{\theta}} \mathrm{Leb}(T \cap P_{\delta}) \leq  \sum_{T \in \mathcal{T}_{\theta}} \mathrm{Leb}(T \cap P_{\delta}) + Cc\delta, \end{displaymath}
so 
\begin{equation}\label{form51} \sum_{T \in \mathcal{T}_{\theta}} \mathrm{Leb}(T \cap P_{\delta}) \gtrsim_{I} \delta, \qquad \theta \in \Theta; \end{equation}
provided that the constant $c = c_{I} > 0$ was chosen sufficiently small. 

Let $\mathcal{T} := \bigcup_{\theta \in \Theta} \mathcal{T}_{\theta}$, and define $G := \bigcup_{T \in \mathcal{T}} T$. We claim that $G$ satisfies \eqref{form48}. Note that the property \eqref{form50} is immediate from the definition of the tubes $\mathcal{T}_{\theta}$.

To prove \eqref{form48}, for $\theta \in \Theta$ fixed let 
\begin{displaymath} P_{\theta} := P \cap \bigcup_{T \in \mathcal{T}_{\theta}} 2T. \end{displaymath}
Then, $P_{\delta} \cap \bigcup_{T \in \mathcal{T}_{\theta}} T \subset P_{\theta,\delta}$ by the triangle inequality, so 
\begin{displaymath} \delta^{2}|P_{\theta}| \gtrsim \mathrm{Leb}(P_{\theta,\delta}) \geq \sum_{T \in \mathcal{T}_{\theta}} \mathrm{Leb}(T \cap P_{\delta}) \stackrel{\eqref{form51}}{\gtrsim_{I}} \delta \quad \Longrightarrow \quad |P_{\theta}| \gtrsim_{I} \delta^{-1} \sim_{I} |P| \end{displaymath}
for each $\theta \in \Theta$. This implies
\begin{displaymath} \sum_{p \in P} |\{\theta \in \Theta : p \in P_{\theta}\}| = \sum_{\theta \in \Theta} |P_{\theta}| \gtrsim_{I} |\Theta||P|. \end{displaymath}
Thus, there exists a subset $P' \subset P$ with $|P'| \sim_{I} |P|$ such that 
\begin{displaymath} |\{\theta \in \Theta :  p \in P_{\theta}\}| \gtrsim_{I} |\Theta|, \qquad p \in P'. \end{displaymath}
For each $p \in P'$, let $\mathcal{T}(p) := \{T \in \mathcal{T} : p \in 2T\}$. Then $|\mathcal{T}(p)| \gtrsim |\{\theta \in \Theta : p \in P_{\theta}\}| \sim_{I} |\Theta|$ for all $p \in P'$. Recall that $P$  (therefore $P'$) is a Frostman $(\delta,1,O_{I}(1))$-set, and $\Theta$ is a Frostman $(\delta,\tfrac{1}{3},\delta^{-2\zeta})$-set. For $C \geq 1$ fixed, these facts imply by Theorem \ref{t:FurstenbergSets} that $\mathcal{T}$ contains a Frostman $(\delta,1,\delta^{-\epsilon/C})$-set, provided $\zeta = \zeta(\epsilon,C) > 0$ was chosen small enough (in other words $\mathcal{T}$ is essentially a ``Kakeya set of tubes"). Now, it follows from C\'ordoba's $L^{2}$-argument (see \cite[Lemma 3.38]{OrponenIncidenceGeometry} for details) that $G = \bigcup_{T \in \mathcal{T}} T$ has measure $\geq \delta^{\epsilon/4}$.

\section{Application to Bochner--Riesz}
\label{sec: bochner--riesz}
In this section, we prove Theorem \ref{theorem: A} and Theorem \ref{theorem: B}. First, we recall some relevant definitions. All these definitions can be found in \cite{cladek_BR,roy,SZ}. Let $\Omega$ be a convex domain as defined in \S\ref{sec: introduction}. 

\begin{definition}
    [Supporting lines]
    A line $\ell$ is called a supporting line for $\Omega$, if $\ell$ contains a point on $\partial\Omega$, and $\Omega$ lies entirely on one of the two closed half-planes determined by $\ell$. If there exists a unique supporting line for $\Omega$ at $\xi$, we call it a tangent line. If $\ell$ is a supporting line for $\Omega$ at the point $\xi\in\partial\Omega$, we let $\omega(\ell)\in\mathbb{S}^1$ be the vector that is parallel to $\ell$, which has $-\xi+\Omega$ to the left of it. We denote the collection of all supporting lines for $\Omega$ at the point $\xi\in\partial\Omega$ by $\mathscr{L}(\Omega;\xi)$.
\end{definition}

\begin{definition}
    [$\delta$-caps]
For $\delta>0$, we define a $\delta$-cap to be a subset of $\partial\Omega$ of the form $$B(\ell,\delta):=\{P\in\partial\Omega: \text{dist}(P,\ell)<\delta\},$$ where $\ell$ is any supporting line for $\partial\Omega$. 
\end{definition}
\begin{definition}
    [Affine dimension]
    Let $N(\Omega,\delta)$ be the minimum number of $\delta$-caps required to cover $\partial\Omega$. The \textit{affine dimension} of $\Omega$ is defined as $$\kappa_\Omega:=\limsup_{\delta\to 0^+}\frac{\log{N(\Omega,\delta)}}{\log{\delta^{-1}}}.$$ 
\end{definition}
\begin{definition}
    [Additive energy]\label{def: additive energy}
    Let $\mathfrak{B}_\delta$ be a collection of $\delta$-caps covering $\partial\Omega$. Consider a partition of $\mathfrak{B}_\delta$ as follows
\begin{equation}
    \label{eqn: partition of caps}\mathfrak{B}_\delta=\bigsqcup_{i=1}^{M_0}\mathfrak{B}_{\delta,i},
\end{equation}
with the following property 

\begin{equation}
    \label{eqn: finite overlap property}
    \forall i,\text{ each $\xi\in\R^2$ lies in at most $M_1$ sets in the class }\{B_{1,i}+\dots+B_{m,i}:B_{j,i}\in\mathfrak{B}_{\delta,i}\}. 
\end{equation}
Let $$\Xi_{\mathfrak{B}_\delta,m}:=\min\{M_0^{2m}\cdot M_1:\text{there is a partition }\eqref{eqn: partition of caps}\text{ of $\mathfrak{B}_\delta$ satisfying }\eqref{eqn: finite overlap property}\},$$ and $$\Xi_m(\Omega,\delta):=\min\{\Xi_{\mathfrak{B}_\delta,m}:\mathfrak{B}_\delta \text{ is a covering of $\partial\Omega$ by $\delta$-caps}\}.$$
Define the \textit{$m$-additive energy} of $\Omega$ to be $$\mathcal{E}_{m,\Omega}:=\limsup_{\delta\to 0^+}\frac{\log\Xi_m(\Omega,\delta)}{\log{\delta^{-1}}}.$$
\end{definition}
Given a convex domain $\Omega$, we associate a direction-set $\Theta(\Omega)$ to it as follows. Fix $\xi\in\partial\Omega$. If $\Omega$ does not have a tangent at $\xi$, then $\mathscr{L}(\Omega;\xi)$ is infinite. Let $(\ell_L,\ell_R)$ be the unique pair of supporting lines such that 
    $$|\omega(\ell_R)-\omega(\ell_L)|=\max_{\ell,\ell'\in\mathscr{L}(\Omega;\xi)}|\omega(\ell)-\omega(\ell')|,$$ and $\omega(\ell_R)$ is to the left of $\omega(\ell_L)$. To make the dependence on $\xi$ explicit, denote $\omega_L(\xi):=\omega(\ell_L)$ and $\omega_R(\xi):=\omega(\ell_R)$. If $\Omega$ has a tangent line $\ell$ at $\xi$, we denote $\omega_R(\xi):=\omega(\ell)=:\omega_L(\xi)$. The reason behind the subscripts `$L$' and `$R$' is the following. After rotation, if $\Omega$ was locally graph-parametrized by a (convex) function $\gamma$ around $\xi$, then the slopes of $\omega_L(\xi)$ and $\omega_R(\xi)$ are $\gamma_L'(\xi_1)$ and $\gamma_R'(\xi_1$), respectively, recalling that the left and right-derivatives of a convex function exist everywhere. When $\gamma_L'(\xi_1)=\gamma_R'(\xi_1)$ i.e., $\gamma'(\xi_1)$ exists, $\Omega$ has a tangent line at $\xi$, and the directions $\omega_L(\xi)$ and $\omega_R(\xi)$ coincide.
   \begin{definition} 
   The direction-set associated with $\Omega$ is defined as
$$
\Theta(\Omega) := \big\{ \omega_L(\xi),\omega_R(\xi) : \xi \in \partial\Omega \big\}\subseteq\mathbb{S}^1.
$$
\end{definition}
The following proposition provides $L^p$-bounds for $B^\alpha_\Omega$ in terms of the additive energy and bounds for the Nikod\'ym maximal function. The result is a reformulation of \cite[Theorem 1.4]{cladek_BR}, which is a generalisation of C\'ordoba's argument for the $L^4$-boundedness of the circular Bochner--Riesz operator \cite{Cordoba}. The same argument was also used by Seeger and Ziesler to prove the $L^4$ bound in Theorem \ref{theorem: seeger--ziesler}. 
\begin{prop}
    \label{prop: 2m bound}
    Let $\Omega$ be a convex domain. Suppose the Nikod\'ym maximal operator associated with the direction-set $\Theta=\Theta(\Omega)$ satisfies the bound $$\|\N_{\Theta;\delta}\|_{L^{\frac{m}{m-1}}(\R^2)\to L^{\frac{m}{m-1}}(\R^2)}\lesssim_\eta\delta^{-\beta-\eta},\quad\text{for all $\eta>0$}.$$ Then $B^\alpha_\Omega$ is bounded on $L^{2m}$ for all $\alpha>\frac{\mathcal{E}_{m,\Omega}}{2m}+\frac{\beta}{2}$.
\end{prop}
We remark that \cite[Theorem 1.4]{cladek_BR} features the $L^{\frac{2m}{2m-1}}$ norm instead of $L^{2m}$. However, these norms are equal, since $B^\alpha_\Omega$ is a self-adjoint operator.

The proofs of Theorem \ref{theorem: A} and Theorem \ref{theorem: B} involve two closely related constructions involving the following generalization of the middle-third Cantor set.

\begin{definition}
    [Homogeneous Moran sets]\label{def: HMS}
    Let $(n_k)_{k\geq 1}$ be a sequence of integers $n_k\geq 2$, and $(c_k)_{k\geq 1}$ be a sequence of `contraction ratios' $c_k\in(0,1)$, satisfying $n_kc_k< 1$ for all $k\geq 1$. Starting from the initial interval $\mathscr{I}_0:=\big\{[-1/2,1/2]\big\}$, let $\mathscr{I}_k$ be given recursively as follows. For each $J\in\mathscr{I}_{k-1}$, let $\mathscr{I}_k(J)$ be a collection of $n_k$ disjoint closed sub-intervals of $J$ (which we call the \textit{children} of $J$), each of length $c_k\cdot\mathcal{L}(J)$, and define $$\mathscr{I}_k:=\bigcup_{J\in\mathscr{I}_{k-1}}\mathscr{I}_k(J),\quad\text{and}\quad E_k:=\bigcup_{I\in\mathscr{I}_k}I.$$ The set $E$ given by $$E:=\bigcap_{k\geq 1}E_k,$$ is called a \textit{homogeneous Moran set} (HMS).
\end{definition}
For an HMS $E$ as above, we call $E_k$ the \textit{$k$th generation} of $E$. For $k \geq 1$, define $\mathscr{I}_k'$ to be the collection of all connected components of $E_{k-1}\setminus E_k$. We refer to these as the intervals removed at generation $k$. Any interval in $\mathscr{I}_{\mathrm{rem}}:=\bigcup_{k\geq 1}\mathscr{I}_k'$ will be called a \textit{removed interval}. We also define $E_{\mathrm{mid}}$ to be the collection of all mid-points of the removed intervals.
\begin{definition}
    [End-point condition]
    Let $E$ be an HMS as defined above. We say that $E$ satisfies the end-point condition, if for all $k \geq 0$, and all $J=[a,b]\in\mathscr{I}_{k}$, the left end-point of the left-most child of $J$ is $a$, and the right end-point of the right-most child of $J$ is $b$.   
\end{definition}
\begin{definition}
    [Generalised Cantor set]\label{def: GCS}
    Let $C$ be an HMS as defined above. We call $C$ a \textit{generalised Cantor set} (GCS), if it satisfies the end-point condition, as well as the Hua--Rao--Wen--Wu \cite{HRWW} condition 
    \begin{equation}
        \label{eqn: HRWW condition}
        \lim_{k \to \infty} \frac{\log c_k}{\log (c_1c_2\cdots c_k)}=0.
    \end{equation}
\end{definition}
 Definition 2.14 from \cite{roy} is a special case of Definition \ref{def: GCS}, where $n_k$ and $c_k$ are constant, and the sets are self-similar. 

The following result gives the quasi-Assouad dimension of generalised Cantor sets in terms of the numbers $n_k$ and $c_k$.
\begin{lemma}\label{lemma: cantor set dimension}
    Let $C \subset [-\tfrac{1}{2},\tfrac{1}{2}]$ be a GCS. Then $$\limsup_{K\to \infty}\frac{\log(n_1n_{2}\dots n_{K})}{-\log(c_1c_{2}\dots c_{K})}\leq\dim_{\mathrm{qA}}(C)\leq\limsup_{K\to \infty}\max_{1\leq k\leq K}\frac{\log(n_kn_{k+1}\dots n_{K})}{-\log(c_kc_{k+1}\dots c_{K})}.$$ 
\end{lemma}
\begin{proof}
    We refer to \cite[Theorem 1.14]{qA_dimension} for the proof.
\end{proof}

We also need the quasi-Assouad dimension of $C_{\mathrm{mid}}$.
\begin{lemma}
    \label{lemma: midpoints dimension}
     Let $C$ be a GCS. Then $$\dim_{\mathrm{qA}}(C_{\mathrm{mid}})=\dim_{\mathrm{qA}}(C).$$
\end{lemma}
\begin{proof}
    Fix $0<r<R \leq 1$. Let $I_R\subseteq [-1/2,1/2]$ be any sub-interval of length $R$. Recall that $C_\mathrm{mid}$ is the collection of all mid-points of all the removed intervals $\I_\mathrm{rem}$. Let us partition  $\I_\mathrm{rem}=\I_\mathrm{rem}^{\leq r}\cup \I_\mathrm{rem}^{>r}$, where $$\I_\mathrm{rem}^{\leq r}:=\{I\in\I_\mathrm{rem}:|I|\leq r\},\quad \I_\mathrm{rem}^{>r}:=\{I\in\I_\mathrm{rem}:|I|>r\}.$$
    Correspondingly, we also define $$C_\mathrm{mid}^{\leq r}:=\{c(I): I\in\I_\mathrm{rem}^{\leq r}\},\quad C_\mathrm{mid}^{>r}:=\{c(I):I\in\I_\mathrm{rem}^{>r}\},$$ where $c(I)$ denotes the centre of $I$.
    For all $I\in\I_\mathrm{rem}^{\leq r}$, we have $\dist(c(I),C)\leq r/2$. As such, 
    \begin{equation}
        \label{eqn: c-mid equation 1}
        |I_R\cap C_\mathrm{mid}^{\leq r}|_{r}\leq |I_R\cap C|_{r/2}.
    \end{equation}
     Since $C_\mathrm{mid}^{>r}$ is a finite set, we can just bound the cardinality of its intersection with $I_R$. Instead of counting the mid-points $c(I)$, however, we count the left end-points $a(I)$ of the relevant intervals $I$. By the end-point condition, $a(I)\in C$ for all $I\in\I_\mathrm{rem}$. On the other hand, $$|a(I)-a(J)|>r\quad\text{for all}\quad I,J\in\I_{\mathrm{rem}}^{>r},\,I\neq J.$$
    Thus, $\{a(I)\}_{I\in\I_\mathrm{rem}^{>r}}$ is an $r$-separated set in $C$. Therefore, 
    \begin{equation}
        \label{eqn: c-mid equation 2}
        |I_R\cap C_\mathrm{mid}^{>r}|\leq |I_R\cap C|_r.
    \end{equation}
    Combining \eqref{eqn: c-mid equation 1} and \eqref{eqn: c-mid equation 2}, it quickly follows that $\qad(C_\mathrm{mid})\leq\qad(C)$.
   
   To prove the converse inequality, note that $C \subset \overline{C}_{\mathrm{mid}}$. Since quasi-Assouad dimension is monotone and stable under closure, we get $\qad(C_{\mathrm{mid}})=\qad(\overline{C}_{\mathrm{mid}})\geq \qad(C)$, completing the proof. \end{proof}

The following is an elementary fact about generalized Cantor sets.
\begin{lemma}
    \label{lemma: limit points}
    Generalised Cantor sets are perfect. More precisely: if $C \subset [-\tfrac{1}{2},\tfrac{1}{2}]$ is a GCS and $x\in C$ is not a left (resp. right) end-point of $J\in\mathscr{I}_k$ for any $k\geq 1$, then there exists a sequence $(x_j)_{j\geq 1}\subset C \, \setminus \, \{x\}$ with $x_j\nearrow x$ (resp. $x_j\searrow x$). 
\end{lemma}

\begin{proof}
   Each point $x\in C$ can be represented as a string $$0.d_1d_2d_3\dots,\quad 0\leq d_k\leq n_k-1.$$ Note that $x$ is a left end-point of some $J\in\mathscr{I}_m$ if and only if $x$ has a representation with $d_k=0$ for all $k\geq m$ (here, we use the end-point condition of $C$). Otherwise, $d_{k_j}\geq 1$ for a subsequence $(d_{k_j})_{j\geq 1}$. Let $x_j\in C$ be the point given by the string $$0.d_1d_2\dots d_{k_j-1}000\dots.$$ Clearly, $x_j<x$ for all $j\geq 1$, and $\lim_j x_j=x$. If $x$ is not a right end-point, we argue similarly. 
\end{proof}
\begin{definition}
    [GCS domains]
    Given a generalized Cantor set $C$, we define the following convex domain 
    $$\Omega(C):=\text{int}(\Omega'(C)),\quad\text{with}\quad\Omega'(C):=\text{conv}\{(t,t^2-1/8):t\in C\},$$ where $\text{conv}(A)$ denotes the convex hull of the set $A$, and  $\text{int}(A)$ denotes the interior of the set $A$. The constant ``$-1/8$" ensures that $0 \in \Omega(C)$, as required by the definition of convex domains.
    
    Let us define the map $F:(\cos t,\,\sin t)\mapsto \tan t$. Then $F(\omega)$ is the slope of a line with direction $\omega\in\mathbb{S}^1$. 
\end{definition}
\begin{prop}\label{prop: slope set of GCS domains}
    Let $C$ be a GCS and let $\Omega=\Omega(C)$. Let $\Theta=\Theta(\Omega)$ be the associated direction-set. Then $F(\Theta)=2(C\cup C_{\mathrm{mid}})\cup\{0\}$. 
\end{prop}
\begin{proof}
    Let us decompose the boundary as 
    \begin{equation}\label{G1G2} \partial\Omega=\mathfrak{G}_1\cup\mathfrak{G}_2, \end{equation}
    where $\mathfrak{G}_1:=\{\xi\in\partial\Omega:\xi_2<1/8\}$, and $\mathfrak{G}_2:=\{\xi\in\partial\Omega:\xi_2= 1/8\}$. Then $\mathfrak{G}_2$ is a line segment, and that $\mathfrak{G}_1$ can be graph-parametrised over the interval $[-1/2,1/2]$, so that there exists a convex function $\gamma:[-1/2,1/2]\to\R$ such that $$\mathfrak{G}_1=\{(t,\gamma(t)):|t|\leq 1/2\}.$$
    By the convexity of $\gamma$, the left and right-derivatives $\gamma_L'$ and $\gamma_R'$ exist everywhere, and for all $\xi\in\mathfrak{G}_1$, we have $F(\omega_L(\xi))=\gamma_L'(\xi_1)$ and $F(\omega_R(\xi))=\gamma_R'(\xi_1)$ (here, we note that $\mathfrak{G}_1$ does not contain the corners $(\pm 1/2,1/8)$).
    
\medskip\noindent\textit{Case 1: $\xi_1\notin C$.} In this case, the graph of $\gamma$ near $\xi$ is a line-segment. If $\xi_1\in I=(a,b)$ for some $I$ removed in the construction of $C$, then
    $$\gamma(t)=a^2-1/8+(a+b)(t-a),\quad\text{for all}\quad t\in I.$$
Then $\gamma'_L(\xi_1)=\gamma_R'(\xi_1)=a+b\in 2\,C_\mathrm{mid}$.

\medskip\noindent\textit{Case 2: $\xi_1\in C$ is not an end-point.}
In this case, the graph near $\xi$ ``looks like" a parabola. Indeed, by Lemma \ref{lemma: limit points}, infinitely many points of the form $(t,t^2-1/8),\,t\in C$ accumulate at $\xi$ from both sides. Let $(x_n)_n\subset C$ be a sequence with $x_n\nearrow \xi_1$. Since $\gamma_R'(\xi_1)$ exists, it is equal to the limit 
$$\gamma_R'(\xi_1)= \lim_n\frac{y_n^2-\xi_1^2}{y_n-\xi_1}=\lim_n y_n+\xi_1=2\xi_1.$$ 
Similarly, there exists a sequence $(y_n)_n\subset C$, $y_n\searrow \xi_1$, and $$\gamma_L'(\xi_1)=\lim_n\frac{\xi_1^2-x_n^2}{\xi_1-x_n}=\xi_1+\lim_nx_n=2\xi_1.$$ Thus, $\gamma_L'(\xi_1)=\gamma_R'(\xi_1)=2\xi_1\in 2C$.

\medskip\noindent\textit{Case 3: $\xi_1\in C$ is an end-point.}  In this case, the graph ``looks like" a parabola on one side of $\xi$ and is a line-segment on the other side. Without loss of generality, let $\xi_1=a$ be a left end-point of the removed interval $(a,b)$. Following the arguments of the previous cases, we find that $\gamma_L'(a)=2\xi_1\in 2C$ and $\gamma_R'(\xi_1)=a+b\in 2\,C_\mathrm{mid}$. 

If $\xi\in\mathfrak{G}_2\setminus\{(-1/2,1/8),(1/2,1/8)\}$, then there is a unique horizontal supporting line at $\xi$, and $\omega_L(\xi)=\omega_R(\xi)=(1,0)$, so the slope is $0$. By the end-point condition, $\pm1/2\in C$, and by Case 3 above, we find that $\gamma_R'(-1/2)=-1$ and $\gamma_L'(1/2)=1$. Thus, $F(\omega_R(-1/2,1/8))=-1\in 2C$ and $F(\omega_L(-1/2,1/8))=0$, while $F(\omega_R(1/2,1/8))=0$ and $F(\omega_L(1/2,1/8))=1\in 2C$. 
\end{proof}
\begin{corollary}\label{cor: direction-set dimension}
    Let $\Omega$ be a domain generated by a GCS, and $\Theta=\Theta(\Omega)$ be the associated direction-set. Then $\qad(\Theta)=\qad(C)$.
\end{corollary}
\begin{proof}
    Note that Definition \ref{def: GCS} does not permit slopes outside the set $[-1,1]$. Recall that $F$ is bi-Lipschitz on the domain $[-\pi/4,\pi/4]$. By Proposition \ref{prop: slope set of GCS domains}, finite stability and bi-Lipschitz invariance of quasi-Assouad dimension, we have $$\qad(\Theta)=\qad(F(\Theta))=\max\{\qad(C),\qad(C_{\mathrm{mid}})\}=\qad(C),$$ by Lemma \ref{lemma: midpoints dimension}.
\end{proof}

We next define a cover of the boundary of a GCS domain by $\delta$-caps, denoted $\mathfrak{B}_{\delta}$, which will be used in the proofs of Proposition \ref{prop: dimension of GCS domain} and Proposition \ref{prop: additive energy of GCS domain} below. First, we make the following observation. Let $\xi,\zeta\in\partial\Omega$ be points such that $\xi_1,\zeta_1\in C$ and $\zeta$ is not the end-point of any removed interval. Then there is a unique (tangent) supporting line $\ell$ at $\zeta$ of slope $2\zeta_1$, and \begin{equation}\label{distanceFormula} \dist(\xi,\ell)\sim (\xi_1-\zeta_1)^2.\end{equation} 
We define $\mathfrak{B}_\delta$ now.
\begin{definition}[Cover $\mathfrak{B}_{\delta}(\eta)$]\label{def:canonicalCover} Let $C$ be a GCS with parameters $c_{k},n_{k}$, and fix $\eta > 0$. Write $\partial \Omega = \mathfrak{G}_{1} \cup \mathfrak{G}_{2}$ as in \eqref{G1G2}. For an interval $I\subseteq[-1/2,1/2]$, let $B_I$ denote the part of $\mathfrak{G}_1$ over $I$. Let $\delta \in (0,1]$ be sufficiently small in terms of $\eta$ (as will be clarified in a moment), and write 
\begin{displaymath} K(\delta) := \max\{k \in \mathbb{N} : c_{1}\cdots c_{k} \geq \delta^{1/2}\} =\max\{k\in\mathbb{N}:|I|\geq \delta^{1/2},\,\forall I\in\I_k\}. \end{displaymath}
Let us consider the following covering of $[-1/2,1/2]$:
\begin{equation}
    \label{eqn: I(delta)}
    \I(\delta):=\I_{K(\delta)}\cup\I_1'\cup\dots\cup \I_{K(\delta)}'.
\end{equation}
Clearly, then $\{B_I:I\in\I(\delta)\}$ is a covering of $\mathfrak{G}_1$. Moreover, if $I \in \mathscr{I}_{k}'$, then $B_I$ is a single line-segment, and so it is a $\delta$-cap. However, $B_I$ need not be a $\delta$-cap when $I\in\I_{K(\delta)}$. By condition \eqref{eqn: HRWW condition}, it holds
$$c_1\dots c_k\leq (c_1\dots c_{k+1})^{1-\eta}\quad\text{for all $k$ sufficiently large}.$$
In particular, $\delta^{1/2} \leq |I| \leq \delta^{1/2 - \eta}$ for all $I \in \I_{K(\delta)}$, provided $\delta > 0$ is sufficiently small in terms of $\eta$. But by convexity and \eqref{distanceFormula}, we can partition each $I\in\I_{K(\delta)}$ into a family $\I_{K(\delta)}(I)$ of $O(\delta^{-\eta})$ many sub-intervals, such that $B_J$ is a $\delta$-cap for all $J\in\I_{K(\delta)}(I)$. 
Define $$\mathfrak{B}_{\delta,0}:=\bigcup_{I\in\I_{K(\delta)}}\{B_J: J\in\I_{K(\delta)}(I)\},$$
 and
\begin{displaymath}
       \mathfrak{B}_\delta:=\mathfrak{B}_{\delta,0}\cup\mathfrak{B}_{\delta,1}\cup\dots\cup\mathfrak{B}_{\delta,K(\delta)}.
    \end{displaymath} 
    Thus, $\mathfrak{B}_{\delta}$ is a cover of $\mathfrak{G}_{1}$ by $\delta$-caps. We also add a single cap to $\mathfrak{B}_{\delta}$ (without introducing additional notation) to cover the ``ceiling" $\mathfrak{G}_{2}$ of $\partial \Omega$. Finally, we remark that $\mathfrak{B}_{\delta} = \mathfrak{B}_{\delta}(\eta)$ actually depends on the parameter $\eta > 0$ (and given $\eta > 0$, the family $\mathfrak{B}_{\delta}$ is well-defined for all $\delta > 0$ small enough).
    
     \end{definition}

\begin{prop}\label{prop: dimension of GCS domain}
    Let $\Omega=\Omega(C)$ be a GCS domain. Then $\kappa_\Omega=\frac12\,\overline{\dim}_B(C)=\frac12\limsup_K\frac{\log(n_1\dots n_K)}{-\log(c_1\dots c_K)}$.
\end{prop}
\begin{proof}
  Let $\delta\in(0,1)$. We want to estimate $N(\Omega,\delta)$, which is the minimum number of $\delta$-caps required to cover $\partial\Omega$.
  Define $K(\delta) := \max\{k \in \mathbb{N} : c_{1}\cdots c_{k} \geq \delta^{1/2}\} =\max\{k\in\mathbb{N}:|I|\geq \delta^{1/2},\,\forall I\in\I_k\}$. We claim that $$|\I_{K(\delta)}|\lesssim N(\Omega,\delta)\lesssim_\eta \delta^{-\eta}|\I_{K(\delta)}|\quad\text{for all $\eta>0$}.$$ This follows from a slight modification of the argument in the proof of \cite[Claim 3.1]{roy}, which was done in the case of a self-similar GCS. We reproduce the proof in the general setting here, omitting a few details which the reader can find in \cite[\S3.2]{roy}. 
  
  It follows from \eqref{distanceFormula}, that in order to cover $\partial\Omega$, we need at least one $\delta$-cap between such $\xi$ and $\zeta$, whenever $|\xi_1-\zeta_1|\geq c\delta^{1/2}$ for a suitable constant $c=O(1)$. Since $|I|\geq \delta^{1/2}$ for all $I\in\I_{K(\delta)}$, it follows that $N(\Omega,\delta)\gtrsim c^{-1}|\I_{K(\delta)}|$. 

  For the upper bound, we fix $\eta > 0$, and use the covering $\mathfrak{B}_{\delta} = \mathfrak{B}_{\delta}(\eta)$ introduced in Definition \ref{def:canonicalCover}. Noting that $|\mathscr{J}_{k}'| \leq |\mathscr{J}_{K(\delta)}|$ for all $1 \leq k \leq K(\delta)$, we obtain
$$|\mathfrak{B}_\delta|\leq|\mathfrak{B}_{\delta,0}|+|\mathfrak{B}_{\delta,1}|+\dots+|\mathfrak{B}_{\delta,K(\delta)}|\lesssim_\eta K(\delta)\delta^{-\eta}|\I_{K(\delta)}|.$$ Since $c_k<1/2$, we also have $K(\delta)\leq \frac12\log\delta^{-1}$. Therefore, $N(\Omega,\delta)\lesssim_\eta\delta^{-\eta}|\I_{K(\delta)}|$. Now $$\log|\I_{K(\delta)}|=\log(n_1\dots n_{K(\delta)})\leq\log\delta^{-1/2}\frac{\log(n_1\dots n_{K(\delta)})}{-\log(c_1\dots c_{K(\delta)})}.$$ On the other hand, by condition \eqref{eqn: HRWW condition}, we have $$\log|\I_{K(\delta)}|\geq (1-\epsilon)\log\delta^{-1/2}\frac{\log(n_1\dots n_{K(\delta)})}{-\log(c_1\dots c_{K(\delta)})},$$ for all $\delta$ sufficiently small. Therefore, $$\kappa_\Omega=\limsup_{\delta\to 0^+}\frac{N(\Omega,\delta)}{\log\delta^{-1}}=\frac12\limsup_{K \to \infty} \frac{\log(n_1\dots n_{K})}{-\log(c_1\dots c_{K})}=\frac12\,\overline{\dim}_B(C),$$ where the last equality follows from \cite[Theorem 1.3]{boxdimformula}.  \end{proof}


\begin{lemma}
    \label{lemma: additive energy of GCS}
    Let $C$ be a GCS and suppose for all $k\geq 1$ there exists a collection of intervals $\J_k$ such that
 \begin{equation}\label{form62} \I_k(J)=\J_k\quad\text{for all}\quad J\in\I_{k-1}. \end{equation} 
    Denote $$g_{m,k}:=\max_{y\in\R}\;\;\Big|\Big\{(I_1,\dots,I_m)\in \J_k^m:y\in I_1+\dots+I_m\Big\}\Big|.$$ Then $$\max_{y\in\R}\Big|\Big\{(I_1,\dots,I_m)\in \I_k^m:y\in I_1+\dots+I_m\Big\}\Big|\leq g_{m,1}\dots g_{m,k},$$ for all $k\geq 1$.
\end{lemma}
\begin{proof}
    Suppose the claim holds for some $k\geq 1$. Let us fix $y\in\R$. Each $I\in\I_{k+1}$ is contained in a unique $J\in\I_{k}$, and by the induction hypothesis, 
    $$\Big|\Big\{(J_1,\dots,J_m)\in \I_k^m:y\in J_1+\dots+J_m\Big\}\Big|\leq g_{m,1}\dots g_{m,k}.$$ Thus, it suffices to show that for all tuples $(J_1,\dots, J_m)\in\I_k^m$,
    $$\Big|\Big\{(I_1,\dots,I_m)\in \I_{k+1}^m:y\in I_1+\dots+I_m,\,I_j\subset J_j\Big\}\Big|\leq g_{m,k+1}.$$ For a bounded interval $I$, let $c(I)$ denote its centre. Now $$y\in I_1+\dots+I_m\iff y'\in (I_1-c(J_1))+\dots+(I_m-c(J_m)),$$ where $y':=y-c(J_1)-\dots-c(J_m)$. Since the intervals $I_j-c(J_j)$ are all translates of elements in $\J_{k+1}$, it follows that $$\Big|\Big\{(I_1,\dots,I_m)\in \J_k^m:y'\in (I_1-c(J_1))+\dots+(I_m-c(J_m))\Big\}\Big|\leq g_{m,k+1},$$ which completes the induction step. \end{proof}
    
\begin{prop}\label{prop: additive energy of GCS domain}
    Let $C$ be a GCS satisfying \eqref{form62}, and let $\Omega=\Omega(C)$. Then $$\mathcal{E}_{m,\Omega}\leq \left(\sup_{k \in \mathbb{N}} \log g_{m,k} \right)\limsup_{\delta\to 0^+}\frac{K(\delta)}{\log\delta^{-1}},$$ where $K(\delta):=\max\{k\in\mathbb{N}:|I|\geq\delta^{1/2},\,\forall I\in\I_k\}$.
\end{prop}
\begin{proof}
    This follows essentially from the argument of \cite[Claim 3.3]{roy}, in the case of self-similar GCS. We provide a modified version of this argument, excluding a few minor details that the reader can find in \cite[\S3.3]{roy}. 
    
    Suppose $\sup_k g_{m,k}=:g_m<\infty$ (otherwise, there is nothing to prove). Fix $\eta > 0$, and recall the cover $\mathfrak{B}_\delta = \mathfrak{B}_{\delta}(\eta)$ from Definition \ref{def:canonicalCover}. Let us fix $0\leq i\leq K(\delta)$. For $B_{j,i}\in\mathfrak{B}_{\delta,i}$, we have 
    \begin{equation}\label{form63} \xi\in B_{1,i}+B_{2,i}+\dots+B_{m,i}\implies \xi_1\in \mathrm{proj}(B_{1,i})+\mathrm{proj}(B_{2,i})+\dots+\mathrm{proj}(B_{m,i}), \end{equation}
    where `$\mathrm{proj}$' denotes projection onto the horizontal coordinate axis. 
    Next, we record that 
    $$\max_{y\in\R}\Big|\Big\{(I_1',\dots,I_m')\in (\I_k')^m:y\in I_1'+\dots+I_m'\Big\}\Big|\leq n_k^mg_m^{k-1},\quad\text{for all $k\geq 1$},$$
    since $\max_{y} |\{(I_{1},\ldots,I_{m}) \in (\mathscr{J}_{k - 1})^{m} : y \in I_{1} + \ldots + I_{m}\}| \leq g_{m}^{k - 1}$ by Lemma \ref{lemma: additive energy of GCS}, and each $I_{j} \in \mathscr{J}_{k - 1}$ contains $\leq n_{k}$ intervals in $\mathscr{J}_{k}'$. Recalling the choice of the families $\mathfrak{B}_{\delta,i}$ for $1 \leq i \leq K(\delta)$ (they were defined as $\delta$-caps over "removed" intervals), it follows from the estimate above, and \eqref{form63}, that
    \begin{equation}\label{form64} \max_{\xi\in\R^2}\Big|\Big\{(B_{1,i},\dots,B_{m,i})\in\mathfrak{B}_{\delta,i}^m:\xi\in B_{1,i}+\dots+B_{m,i}\Big\}\Big| \leq n_{K(\delta)}^mg_m^{K(\delta)},\quad\text{for all}\quad 1\leq i\leq K(\delta). \end{equation}
    For $i = 0$, the collection $\mathfrak{B}_{\delta,i}$ was defined a little differently, namely by selecting $\lesssim_{\eta} \delta^{-\eta}$ many $\delta$-caps corresponding to each interval in $\mathscr{J}_{K(\delta)}$. Now, applying Lemma \ref{lemma: additive energy of GCS} with $k = K(\delta)$, we obtain the following analogue of \eqref{form64} in the case $i = 0$:
    \begin{equation}\label{form65} \max_{\xi\in\R^2}\Big|\Big\{(B_{1,0},\dots,B_{m,0})\in\mathfrak{B}_{\delta,0}^m:\xi\in B_{1,0}+\dots+B_{m,0}\Big\}\Big| \lesssim_{\eta} \delta^{-\eta m} g_m^{K(\delta)},\quad\text{for all}\quad 1\leq i\leq K(\delta). \end{equation}
    
    Recall Definition \ref{def: additive energy}, which we now apply with $M_{0} := K(\delta) + 1$ and $M_{1} := \max\{n_{K(\delta)}^{m},O_{\eta}(\delta^{-\eta m})\} \cdot g_{m}^{K(\delta)}$, thanks to \eqref{form64}-\eqref{form65}. It follows that 
    $$\Xi_m(\Omega,\delta)\leq\Xi_{\mathfrak{B}_\delta,m}\lesssim_\eta K(\delta)^{2m} \cdot \delta^{-\eta m}n_{K(\delta)}^mg_m^{K(\delta)},$$ so that $$\mathcal{E}_{m,\Omega}\leq 2m\frac{\log K(\delta)}{\log\delta^{-1}}+m\frac{\log n_{K(\delta)}}{\log\delta^{-1}}+\log g_m\limsup_{\delta\to 0^+}\frac{K(\delta)}{\log\delta^{-1}} + \eta m$$
    As noted earlier, $K(\delta)\leq\frac12\log\delta^{-1}$, so the first limit is zero. By \eqref{eqn: HRWW condition}, the following holds, provided $K(\delta)$ is large enough: $$n_{K(\delta)}\leq c_{K(\delta)}^{-1}\leq (c_1\dots c_{K(\delta)})^{-\eta}\leq\delta^{-\eta}.$$ Finally, let $\eta \to 0^+$ to obtain the desired bound.
\end{proof}

Lastly, we need the following result about the existence of large interval families with small additive energy. The domains appearing in Theorem \ref{theorem: B} and Theorem \ref{theorem: A} will be GCS domains constructed using these intervals.
\begin{lemma}[Lemma 2.11 \cite{roy}]\label{lemma: interval family}
 For $m\geq 2$ and all $N$ sufficiently large, there exists a family $\mathscr I(N;2m)$ of $N$ subintervals of $[-1/2,1/2]$ satisfying the following properties. 
 \begin{enumerate}[label=(\roman*)]
     \item Each interval in $\I(N;2m)$ has length $N^{-m}$.
     \item The intervals in $\I(N;2m)$ are separated by $(m/2)N^{-m}$: $$\dist(I_1,I_2)\geq (m/2)N^{-m}\quad\text{for all}\quad I_1,I_2\in\I(N;2m),\;I_1\neq I_2.$$ 
    \item There exists a constant $g_m\geq 1$, depending only on $m$, such that $$\max_{y\in\R}\;\;\Big|\Big\{(I_1,\dots,I_m)\in \I(N;2m)^m:y\in I_1+\dots+I_m\Big\}\Big|\leq g_m.$$
     \item The family $\I(N;2m)$ contains the intervals $I^-:=[-1/2,-1/2+N^{-m}]$ and $I^+:=[1/2-N^{-m},1/2]$. 
 \end{enumerate}
 \end{lemma}
Let us now prove the results stated in \S\ref{sec: BR intro}.
\B*
\begin{proof}
    The domain $\Omega$ is a minor modification of the one constructed in \cite[p.~11]{cladek_BR}. We define an HMS as follows. Let $\mathscr{I}_0:=\big\{[-1/2,1/2]\big\}$, and $\mathscr{I}_1=\mathscr{I}(2;3)$. For $I\in\mathscr{I}_k$, let $\mathcal{L}_I$ be the order-preserving affine transformation mapping $I$ onto $[-1/2,1/2]$. We define $\mathscr{I}_{k+1}$ as $$\mathscr{I}_{k+1}:=\{\mathcal{L}_I^{-1}(J)\subset I:J\in\mathscr{I}(2^{k};3),\; I\in\mathscr{I}_k\}.$$ Then as in Definition \ref{def: HMS}, we define $$C_k:=\bigcup_{I\in\mathscr{I}_k}I,\quad\text{and}\quad C:=\bigcap_{k\geq 1}C_k.$$ Then $C$ is an HMS with $n_k=2^{k}$, and $c_k=2^{-3k}$ for all $k\geq 1$. Then $$\lim_k\frac{\log c_k}{\log (c_1\dots c_k)}=\lim_k\frac{2k}{k(k+1)}=0.$$
    Property (iv) in Lemma \ref{lemma: interval family} ensures that $C$ satisfies the end-point condition. We let $\Omega=\Omega(C)$ be the domain generated by $C$. Observe that \begin{equation}
    \label{eqn: dimension construction B}
    \lim_K\frac{\log(n_kn_{k+1}\dots n_K)}{-\log(c_kc_{k+1}\dots c_K)}=\frac13\quad\text{for all}\quad 1\leq k\leq K.
\end{equation} In particular for $k=1$, Proposition \ref{prop: dimension of GCS domain} yields $\kappa_\Omega=\frac 16$.
    
    Now we show that $\|B^\alpha_\Omega\|_{L^6\to L^6}<\infty$ for all $\alpha > 0$. First, we estimate $\mathcal{E}_{3,\Omega}$. By property (iii) in Lemma \ref{lemma: interval family}, we have $\sup_k g_{3,k}\leq g_3$. On the other hand, $K(\delta)$ satisfies $$c_1^{-1}\dots c_{K(\delta)}^{-1}\leq\delta^{-1/2}\implies K(\delta)\lesssim(\log\delta^{-1/2})^{1/2}.$$ By Proposition \ref{prop: additive energy of GCS domain}, it follows that $\mathcal{E}_{3,\Omega}=0$. 
    
    Next, consider the direction-set $\Theta=\Theta(\Omega)$. By Corollary \ref{cor: direction-set dimension}, Lemma \ref{lemma: cantor set dimension}, and \eqref{eqn: dimension construction B}, we find that $$\qad(\Theta)=\qad(C)=1/3\leq 1/2.$$ By Proposition \ref{prop: main}, we get  $$\|\N_{\Theta;\delta}f\|_{L^{3/2}(\R^2)}\lesssim_\eta\delta^{-\eta}\|f\|_{L^{3/2}(\R^2)},\quad\text{for all $\eta>0$}.$$
   Thus, the parameter $\beta$ from the statement of Proposition \ref{prop: 2m bound} is $0$. Using the same proposition, it follows that $\|B^\alpha_\Omega\|_{L^6\to L^6}<\infty$ for all $\alpha>0$. Hence, $p_\Omega\geq 6$.
\end{proof}

We finally prove Theorem \ref{theorem: A}, using a variation of the argument above. The main difference is that the Cantor set underlying the construction of the domain $\Omega$ is self-similar in the case of Theorem \ref{theorem: A} --  unlike in the proof Theorem \ref{theorem: B}. The self-similarity of $C$ allows us to apply results from \cite{roy} to obtain non-trivial $L^{p}$-bounds for $B_{\Omega}^{\alpha}$ for $p > 6$. However, the statement at $p = 6$ becomes slightly weaker. It is conceivable that the self-similarity of the Cantor set is not essential in \cite{roy}; if this is the case, then a non-self-similar construction might yield a version of Theorem \ref{theorem: A} which supersedes Theorem \ref{theorem: B}.
\A*
\begin{proof}
Fix $\epsilon>0$ as in the statement. The domain $\Omega$ will be precisely the domain $\Omega_p$ constructed in \cite[\S3.1]{roy} corresponding to $p=6$. Nevertheless, for completeness, we recall the construction here. Let $N$ be a large integer, to be specified later. We define an HMS as follows. Let $\mathscr{I}_0:=\big\{[-1/2,1/2]\big\}$, and $\mathscr{I}_1=\mathscr{I}(N;6)$. We define $\mathscr{I}_{k+1}$ as $$\mathscr{I}_{k+1}:=\{\mathcal{L}_I^{-1}(J)\subset I:J\in\mathscr{I}(N;3),\; I\in\mathscr{I}_k\}.$$ We define $$C_k:=\bigcup_{I\in\mathscr{I}_k}I,\quad\text{and}\quad C:=\bigcap_{k\geq 1}C_k.$$ Then $C$ is an HMS with $n_k=N$ and $c_k=N^{-3}$ for all $k\geq 1$, thus $$\lim_k\frac{\log c_k}{\log (c_1\dots c_k)}=\lim_k \frac{1}{k}=0.$$ Again, property (iv) in Lemma \ref{lemma: interval family} ensures that $C$ satisfies the end-point condition, and so it is a GCS. Let $\Omega=\Omega(C)$. Observe that 
\begin{equation}
    \label{eqn: dimension construction A}
    \lim_K\frac{\log(n_kn_{k+1}\dots n_K)}{-\log(c_kc_{k+1}\dots c_K)}=\frac13\quad\text{for all}\quad 1\leq k\leq K.
\end{equation}
In particular, for $k=1$, Proposition \ref{prop: dimension of GCS domain} gives $\kappa_\Omega=\frac16$.

Now we prove the relevant estimates for $B^\alpha_\Omega$. We do so by interpolating the estimates at $p=6,\,18,$ and $\infty$. The $p=\infty$ bound follows directly from the $L^\infty$ bound in Theorem \ref{theorem: seeger--ziesler}. The $p=18$ bound is precisely the $L^{6m}$ bound in \cite[Theorem 1.3]{roy} for $m=3$, which was obtained using decoupling estimates (see \cite[Proposition 5.7]{roy} for the details). Thus, we only need to prove the $L^6$ bound. We establish this from Proposition \ref{prop: 2m bound} with $m=3$. To estimate $\mathcal{E}_{3,\Omega}$, we see that due to property (iii) in Lemma \ref{lemma: interval family}, $\sup_k g_{m,k}\leq g_3$. Next, we have $$c_1^{-1}\dots c_{K(\delta)}^{-1}\leq\delta^{-1/2}\implies K(\delta)\leq\frac16\frac{\log\delta^{-1}}{\log N}.$$ By Proposition \ref{prop: additive energy of GCS domain}, it follows that $\mathcal{E}_{3,\Omega}\leq 6\epsilon$ provided we choose $N\geq g_3^{\frac{1}{36\epsilon}}$. By Corollary \ref{cor: direction-set dimension}, Lemma \ref{lemma: cantor set dimension}, and \eqref{eqn: dimension construction A}, we have $$\qad(\Theta)=\qad(C)=1/3\leq 1/2.$$ Thus by Proposition \ref{prop: main}, $$\|\N_{\Theta;\delta}f\|_{L^{3/2}(\R^2)}\lesssim_\eta\delta^{-\eta}\|f\|_{L^{3/2}(\R^2)},\quad\text{for all $\eta>0$}.$$ Applying Proposition \ref{prop: 2m bound} with $\beta=0$, we get $\|B^\alpha_{\Omega}\|_{L^6\to L^6}<\infty$ for all $\alpha>\epsilon$.
\end{proof}

\begin{remark}
    Both Theorem \ref{theorem: B} and Theorem \ref{theorem: A} hold for domains that follow a very specific construction scheme. It would be interesting to study the Bochner--Riesz problem for more general classes of convex domains. For instance, it is conceivable that for all $\Omega$ such that $\Theta(\Omega)$ is an $s$-Ahlfors regular direction set, the bounds in Theorem \ref{theorem: seeger--ziesler} can be improved. Indeed, the sharp examples of Theorem \ref{theorem: seeger--ziesler} discussed in \cite{SZ, cladek_BR} involve direction-sets containing long arithmetic progressions at every scale. The Ahlfors regularity condition would rule out these examples.
\end{remark}
\appendix
\section{Kakeya maximal estimate}\label{app:kakeya}
Here we give a proof of Theorem \ref{theorem: fractal kakeya}. First, we need the following bound for $\delta$-tubes.
\begin{lemma}\label{lemma: rogers}
    Let $\mathcal{T}$ be a collection of $\delta$-tubes whose directions form a Frostman $(\delta,s,C)$-subset of $\mathbb{S}^{1}$, and such that there is only one tube in each direction. Then
\begin{displaymath} \Big\| \sum_{T \in \mathcal{T}} \chi_T \Big\|_{L^{p'}(\R^2)} \lessapprox C^{1/p}\delta^{2/p'}|\mathcal{T}|,\quad p=1+s.\end{displaymath} 
\end{lemma}
The following short proof was communicated to us by Keith Rogers.
\begin{proof} 
  We start by writing 
  $$\Big\| \sum_{T \in \mathcal{T}} \chi_T \Big\|_{L^{(s + 1)/s}(\R^2)}^{(s + 1)/s}=\int\Big(\sum_{T\in\T}\chi_T\Big)\cdot\Big(\sum_{T'\in\T}\chi_{T'}\Big)^{1/s}=\sum_{T\in\T}\int_T\Big(\sum_{T'\in\T}\chi_{T'}\Big)^{1/s}.$$
  Applying Minkowski's inequality with exponent $1/s \geq 1$, we get
\begin{align*}\label{eqn: minkowski step} \nonumber\sum_{T \in \mathcal{T}} \int_{T} \Big( \sum_{T' \in \mathcal{T}} \chi_{T'} \Big)^{1/s} \,& \leq \sum_{T \in \mathcal{T}} \Big( \mathop{\sum_{T' \in \mathcal{T}}}_{T \cap T' \neq \emptyset} |T \cap T'|^{s} \Big)^{1/s}\\
& \lesssim \sum_{T \in \mathcal{T}} \Big( \mathop{\sum_{T' \in \mathcal{T}}}_{T \cap T' \neq \emptyset} \frac{\delta^{2s}}{\angle(T,T')^{s}} \Big)^{1/s},\end{align*}
where $\angle(T,T'):=|\omega(T)-\omega(T')|$.
Now
\begin{displaymath} \mathop{\sum_{T' \in \mathcal{T}}}_{T \cap T' \neq \emptyset} \frac{1}{\angle(T,T')^{s}}\lesssim\sum_{k\geq 0}^{\log\delta^{-1}}(2^k\delta)^{-s}|\{T'\in\T:\angle(T,T')\sim 2^k\delta\}|.\end{displaymath} 
By the Frostman $(\delta,s,C)$-set hypothesis on the direction-set, we have $$|\{T'\in\T:\angle(T,T')\lesssim 2^k\delta\}|\lesssim C|\T|(2^k\delta)^s,\quad\text{and so}\quad \mathop{\sum_{T' \in \mathcal{T}}}_{T \cap T' \neq \emptyset} \frac{1}{\angle(T,T')^{s}}\lessapprox C|\T|.$$ Plugging this estimate back above concludes the proof. 
\end{proof}
Let us recall the statement of Theorem \ref{theorem: fractal kakeya}.
\kakeya*
\begin{proof}
        It suffices to prove the estimate for $p=1+s$ and interpolate with the trivial $L^1$-bound.
        Similar to the proof of Proposition \ref{prop: main}, we begin by discretizing the maximal operator, and bound 
        $$\|\K_\delta f\|_{L^p(\dd{\mu})}\lesssim\bigg(\sum_{\omega\in\Xi}|\K_\delta f(\omega)|^p\,\mu(B_\delta(\omega))\bigg)^{1/p},$$ where $\Xi$ is a $\delta$-net in $\supp{\mu}$. Since $|\Xi|\lesssim\delta^{-1}$, and by H\"older's inequality, 
        $$\K_\delta f(\omega)\lesssim\delta^{-1/p}\|f\|_{L^p(\R^2)}\quad\text{for all}\quad\omega\in\Xi,$$
         the contribution from the points $\omega\in\Xi$ for which $\mu(B_\delta(\omega))\leq \delta^{2/p}$ in the above sum is negligible. On the other hand, by the Frostman condition we have $\mu(B_\delta(\omega))\lesssim\delta^s$ for all $\omega\in\Xi$. Therefore, at the cost of a logarithm, we may assume that there exists a $\lambda\in[\delta^{2/p},\delta^s]$ such that $\mu(B_\delta(\omega))\sim\lambda$ for all $\omega\in\Xi$. Next, for each $\omega\in\Xi$, we find a tube $T_\omega$ such that $$\K_\delta f(\omega)\lesssim\fint_{T_\omega}|f|.$$ By the above, and duality, we find an $\ell^{p'}$-normalized non-negative sequence $(a_\omega)_{\omega\in\Xi},$ such that
        \begin{equation}
            \label{eqn: kakeya maximal reduction}
            \|\K_\delta f\|_{L^p(\dd{\mu})}\lessapprox\lambda^{1/p}\delta^{-1}\sum_{\omega\in\Xi}a_\omega\int_{T_\omega}|f|\leq\lambda^{1/p}\delta^{-1}\|\sum_{\omega\in\Xi}a_\omega\chi_{T_\omega}\|_{L^{p'}(\R^2)}\,\|f\|_{L^p(\R^2)}.
        \end{equation}
    Let $r\in[\delta,1]$. For all $r$-ball $B_r$, we have $$r^s\gtrsim\mu(B_r)\gtrsim\sum_{\omega\in B_r}\mu(B_\delta(\omega))\implies |\Xi\cap B_r|\lesssim(\lambda|\Xi|)^{-1}|\Xi|r^s.$$ In other words, $\Xi$ is a Frostman $(\delta,s,(\lambda|\Xi|)^{-1})$-set. By Lemma \ref{lemma: rogers}, 
    $$\|\sum_{\omega\in\Xi}\chi_{T_\omega}\|_{L^{p'}(\R^2)}\lessapprox \lambda^{-1/p}\delta^{2/p'}|\Xi|^{1/p'}.$$ By dyadically pigeonholing in the size of $a_\omega$, we find, 
    $$\|\sum_{\omega\in\Xi}a_\omega\chi_{T_\omega}\|_{L^{p'}(\R^2)}\lessapprox\lambda^{-1/p}\delta^{2/p'},$$ using which in \eqref{eqn: kakeya maximal reduction} yields $$\|\K_\delta f\|_{L^p(\dd{\mu})}\lessapprox\delta^{1-2/p}\|f\|_{L^p(\R^2)},$$ as required.
\end{proof}

\bibliographystyle{plain}
\bibliography{Citation,references}
\end{document}